\documentclass[11pt]{article}
\usepackage{enumerate}
\usepackage{bbm}
\usepackage{amsfonts}
\usepackage[T1]{fontenc}
\usepackage{latexsym,amssymb,amsmath,amsfonts,amsthm}
\usepackage{graphicx, color}
\usepackage{subfigure}
\usepackage{epsf}
\usepackage{ulem}
\usepackage{epstopdf}
\usepackage{amssymb}
\usepackage{mathrsfs}
\usepackage{diagbox}
\usepackage[ruled]{algorithm2e}
\usepackage{placeins}

\newcommand{\R}{{\mathbb R}}

\newcommand{\be}{\begin{eqnarray}}
\newcommand{\ben}{\begin{eqnarray*}}
\newcommand{\en}{\end{eqnarray}}
\newcommand{\enn}{\end{eqnarray*}}
\newcommand{\ba}{\backslash}
\newcommand{\pa}{\partial}

\newcommand{\ov}{\overline}

\newcommand{\Grad}{{\rm Grad\,}}

\newtheorem{theorem}{Theorem}[section]
\newtheorem{lemma}[theorem]{Lemma}

\newtheorem{definition}[theorem]{Definition}

\newtheorem{proposition}[theorem]{Proposition}
\newtheorem{assumption}[theorem]{Assumption}

\begin{document}
\title{\bf Inverse Obstacle Scattering from Multi-Frequency Near-Field Backscattering Data}
\author{Jialei Li\thanks{School of Mathematical Sciences, University of Chinese Academy of Sciences, Beijing 100049, China, and Academy of Mathematics and Systems Science,
Chinese Academy of Sciences, Beijing 100190, China. Email: lijialei21@mails.ucas.ac.cn},\and
Xiaodong Liu\thanks{Corresponding author. Academy of Mathematics and Systems Science,
Chinese Academy of Sciences, Beijing 100190, China. Email: xdliu@amt.ac.cn}
}
\date{}
\maketitle

\begin{abstract}
This paper addresses the inverse obstacle scattering problem of simultaneously reconstructing the obstacle geometry and boundary conditions from multi-frequency near-field backscattering data. We first establish rigorous high-frequency asymptotic expansions for the scattered near-field, leveraging pseudo-differential operators (PDOs) to characterize the interaction between wavefront propagation and obstacle boundaries, where the principal symbol of the PDO governs the leading-order behavior of the scattering field. Based on these asymptotic results, we prove a global uniqueness theorem for the simultaneous recovery of the obstacle shape and impedance boundary condition under convexity assumptions. Furthermore, we develop a three-stage numerical reconstruction framework: (1) qualitative shape reconstruction via the direct sampling method; (2) quantitative boundary refinement via shape optimization; and (3) decoupled reconstruction of the boundary condition.
A highlight of this algorithm is that all the three steps avoid computing the direct problem. Numerical experiments are presented to verify the robustness and efficiency of the proposed algorithm.

\vspace{.2in}

\vspace{.2in} {\bf AMS Subject Classifications:}
35P25, 45Q05, 78A46

\end{abstract}

\section{Introduction}

The inverse scattering problem has been a central topic in applied mathematics and engineering for decades, with applications spanning medical imaging, geophysical exploration, radar, sonar, and non-destructive testing. It concerns the reconstruction of the geometric and physical properties of an unknown obstacle from measurements of scattered waves. Among inverse scattering configurations, backscattering—where transmitters and receivers are co-located—stands out for its practical relevance, as such configurations are easier to implement and naturally arise in many sensing platforms. Nevertheless, backscattering data carry only a small fraction of the information available in full-aperture measurements. As a consequence, establishing rigorous theoretical guaranties for reconstruction is extremely challenging, and existing numerical schemes often rely on restrictive convexity assumptions or a priori knowledge of the boundary conditions of the obstacle \cite{ArensJiLiu2020,Christiansen2013,KressRun98,LiLiu2015,LiLiuWang2017,Shin2016}.

Recently, Rakesh and Uhlmann \cite{RakUhl} proved that backscattering data can identify radial media for angularly controlled potentials; however, the uniqueness for general media remains an open problem \cite{EskinRalston3, EskinRalston2, Lagergren, StefUhl, Uhlmann2001, Wang}. For obstacle scattering, the theoretical development is even more limited due to the strong nonlinearity arising from the unknown boundary and physical properties of the obstacle. These challenges leave both the uniqueness of reconstruction and the stable numerical recovery  as largely open questions.

The mathematical foundation for high-frequency obstacle scattering was established primarily by Majda \cite{Majda1976}, who, using the Lax  parametrix method, constructed an approximate solution to derive the asymptotic leading term of the far-field pattern under plane-wave incidence. In our previous work \cite{LiLiuShi2025}, we extended Majda's far-field asymptotic analysis and uniqueness results to impedance boundary conditions in two dimensions. All these theoretical guaranties, however, rely on the far-field approximation, which assumes the wave source is infinitely far from the obstacle.

In contrast, many real-world applications, such as medical ultrasound probes and radar imaging, operate in the near-field regime with point-source emitters. In this setting, the interaction between the location of the source point, the geometry of the obstacle, and its boundary condition becomes much more sophisticated. Unlike the far-field scenario, where the obstacle’s curvature is fundamentally decoupled from the source location, near-field wave propagation is closely tied to their relative positions.

In this paper, we bridge this theoretical gap by investigating the near-field, point-source inverse obstacle scattering problem. Let $D\subset \mathbb{R}^N$ ($N=2,3$) be a convex bounded obstacle, and denote by $u^i(\cdot;z)$ the incident field generated by a point source at $z\in \mathbb{R}^N\setminus \ov{D}$. The scattered field $u^s$ satisfies the following boundary value problem:
\be \label{eq-scatteringEQs}
\left\{
\begin{aligned}
\Delta u^s(x;k)+k^2 u^s(x; k) &= 0, && \text{in } \mathbb{R}^N\backslash \ov{D},\\
\mathcal{B}u^s &= -\mathcal{B}u^i, && \text{on } \partial D,\\
\lim_{r=\|x\|\to \infty} r^{\frac{N-1}{2}}\left|\frac{\partial u^s}{\partial r} - ik u^s\right| &= 0,
\end{aligned}
\right.
\en
where the boundary operator $\mathcal{B}$ takes one of the following:
\be \label{eq-BCs}
(1)\;\mathcal{B} u = u,\qquad
(2)\;\mathcal{B} u = \frac{\partial u}{\partial n},\qquad
(3)\;\mathcal{B} u = \frac{\partial u}{\partial n} +ik\gamma u,
\en
with $n$ denoting the exterior unit normal to $\partial D$ and $\gamma$ a strictly positive impedance function in $C^\infty(\partial D)$. 
The inverse problem is to recover both the boundary $\partial D$ and the boundary condition $\mathcal{B}$ from multi-frequency near-field backscattering measurements $\{u^s(x;x,k)\}$ collected at coincident source–receiver locations on a measurement curve $\Gamma_R:=\partial B_R$, where $B_R$ is a ball of radius $R$ containing $\ov{D}$. To address this problem,  we extend Majda's high-frequency asymptotic framework to the point-source near-field regime. We derive explicit asymptotic expansions of the scattered field for Dirichlet, Neumann, and Robin obstacles in both two and three dimensions, revealing precisely how the obstacle's curvature couples with the source location. These asymptotics ultimately yield a global uniqueness theorem for the simultaneous reconstruction of the obstacle’s boundary and its associated boundary condition.

Numerically, the stable recovery of the impedance function $\gamma(x)$ poses a significant challenge, as it heavily relies on a highly accurate approximation of the obstacle's shape. To mitigate the extreme sensitivity of impedance evaluation to boundary errors, we propose a shape-impedance decoupling strategy. This approach features an intermediate shape optimization phase (see Step 3 in Section \ref{sec-Numeric}) and ensures robust impedance reconstruction while preserving the computational efficiency of the direct sampling method.

The remainder of the paper is structured as follows: Section \ref{sec-InvP} presents our main asymptotic and uniqueness results; Section \ref{sec:preliminaries} compiles the analytical preliminaries, including asymptotics of the fundamental solution and pseudo-differential operator calculus required for the proofs; Section \ref{Proofs} provides the complete proofs of the main results; Section \ref{sec-Numeric} details the reconstruction algorithm and a new iterative correction method for shape refinement; and Section 6 presents numerical experiments to validate the proposed method.

\section{High-frequency expansions and their applications}\label{sec-InvP}
For any $z\in \R^N\backslash \ov{D}$, we define the illuminated side of $\partial D$ with respect to $z$ by
\ben
\partial D^+_z:=\{y\in\partial D: n(y)\cdot(y-z)<0 \}.
\enn
We further define the illuminated side with respect to two points $x, z \in \R^N\backslash \ov{D}$ by
\ben
\partial D^+_{x,z}:=\partial D^+_z \cap\partial D^+_x,
\enn
as illustrated in Figure \ref{fig:IlluSide}. The non-illuminated side with respect to $x$ and $z$ is $\partial D^-_{x,z}:= \partial D\backslash \partial D^+_{x,z}$.

\begin{figure}
    \centering
    \includegraphics[width=0.5\linewidth]{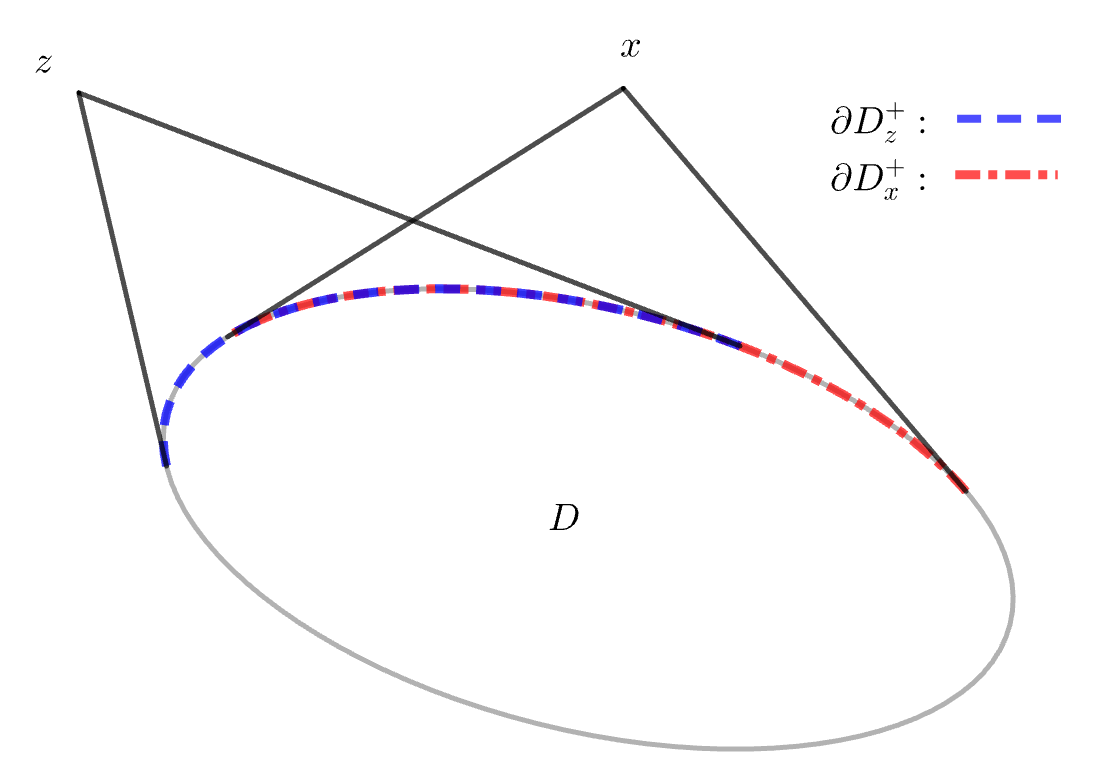}
    \caption{Illustration of illuminated side $\partial D^+_z, \partial D_x^+$.}
    \label{fig:IlluSide}
\end{figure}

The function
\ben
\psi(y):= \|x-y\|+\|y-z\|, \quad y\in \partial D
\enn
represents the total path length of the ray scattered at $y$, which plays a key role in physical optics. 
Recall that the surface gradient and surface Hessian of $\psi$ on $\partial D$ are defined, respectively, by 
\begin{equation*}
    \Grad \psi(y):= P_{\partial D} (y) \nabla \psi(y)\quad\mbox{and}\quad H_{\partial D}\psi(y) := \Grad\Grad^T \psi (y),
\end{equation*}
where $P_{\partial D}(y)$ is the projection operator at $y$. 
A point $y\in \partial D$ is called a stationary point of $\psi$ if $\Grad \psi(y)=0$;  it is referred to as a non-degenerate stationary point if  $\det H_{\partial D}\psi(y)\neq 0$. 
For any $z\in\R^N\ba\ov{D}$, we define 
\ben
\theta_z(y) := \frac{y-z}{\|y-z\|},\quad y\in \partial D.
\enn

Before stating our main results, we present a key geometric lemma and a critical assumption.
\begin{lemma}\label{lemma:UniqueYplus}
If $\partial D_{x,z}^+$ is not empty, there exists a unique stationary point $y^+\in \partial D_{x,z}^+$ of $\psi$, and this point $y^+$ is non-degenerate.
\end{lemma}

\begin{assumption}\label{assumption:geometric}
For the  convex obstacle $D$ and points $x, z\in \mathbb{R}^N\setminus\ov{D}$, all stationary points of $\psi$ are non-degenerate in the non-illuminated side $\partial D^-_{x,z}$.
\end{assumption}

This assumption fails only if a portion of $\partial D$ coincides with an ellipse (or ellipsoid) with foci at $x$ and $z$. For bounded obstacles with analytic boundaries (curves or surfaces), Assumption \ref{assumption:geometric} is always satisfied.

Our first main result is the high-frequency asymptotic expansion of the scattered field for the general near-field configuration:

\begin{theorem}\label{thm:general}
Under Assumption \ref{assumption:geometric}, provided that $\partial D_{x,z}^+$ is nonempty, the scattered field $u^s(x;z,k)$ admits the following asymptotic expansion as $k\to \infty$:
\be \label{eq-Asym-u}
u^s(x;z,k)= A_N(x,z,k)\theta_x^+\cdot n^+ \frac{\gamma^+ +\theta_z^+\cdot n^+}{\gamma^+ -\theta_z^+\cdot n^+}e^{ik\psi^+} + O\left(k^{\frac{N-5}{2}}\right),
\en
where 
\begin{itemize}
    \item $y^+$ is the unique stationary point of $\psi$ in $\partial D^+_{x,z}$;
    \item $f^+:=f(y^+)$ for  $f=\gamma,\theta_z,\theta_x,n$ and $\psi$;
    \item the dimensional factor $A_N(x,z,k)$ is defined as
    \ben
    A_N(x,z,k):=\frac{e^{\frac{3-N}{4}\pi i}}{2k} \left( \frac{k}{2\pi\|x-y^+\|\|y^+ -z\||\det H_{\partial D}\psi(y^+)|} \right)^{\frac{N-1}{2}}.
    \enn
\end{itemize}
\end{theorem}

For the backscattering case ($x=z$)—the focus of many practical applications—we obtain the following simplified result:
\begin{theorem}\label{thm:backscattering}
Under Assumption \ref{assumption:geometric}, the backscattering scattered field admits the following simplified asymptotic expansion:
\be \label{eq-Asym-u0-backsca}
u^s(x;x,k) = -A_N(x,x,k) \frac{\gamma^+-1}{\gamma^++1}e^{2ik\|x-y^+\|}+ O\left(k^{\frac{N-5}{2}}\right),
\en
where $y^+$ denotes the closest point to $x$ on $\partial D$.
\end{theorem}

These results provide explicit high-frequency behavior of the scattered field, which characterizes the geometric dependence of the scattered field on the obstacle's shape and the relative position of the source and observation points. The backscattering formula is particularly useful for inverse problems, as it directly relates the measured field to the distance from the source to the obstacle's boundary.

Notably, the leading term of the above asymptotic expansion vanishes when $\gamma(y^+)=1$. This corresponds to the case where the incident field locally satisfies the boundary condition at $y^+$ in the high-frequency limit. We therefore introduce an additional assumption on the boundary condition:
\begin{assumption}\label{assumption:boundcondi}
    The impedance function satisfies $\gamma= 1$ at at most finitely many points on $\partial D$.
\end{assumption}

As a direct application of the above high-frequency asymptotic analysis of the backscattering data, we state the following global uniqueness result for the inverse obstacle scattering problem.

\begin{theorem} \label{thm:uniqueness}
Under Assumption \ref{assumption:geometric} and \ref{assumption:boundcondi}, the boundary $\partial D$ and the impedance function $\gamma$ are uniquely determined by the near-field backscattering measurements
\ben
\{u^s(x;x,k): x\in \Gamma_R, k\in K\},
\enn
where $K$ can be either an interval $(k_{min},k_{max})$ or a set $\{k_j=k_0+j\delta_k:\,j = 0, 1, \cdots\}$ with $\delta_k<\frac{\pi}{4R}$.
\end{theorem}

\section{Preliminaries}\label{sec:preliminaries}

\subsection{Fundamental solutions}
In the time domain, the distributional fundamental solution of the wave equation $\partial^2_t w - \Delta w=0 $ in $\R^N$ is given by 
\be \label{eq-fundaSolution-timedomain}
\hat{G}(y,t;z) =\left\{
\begin{aligned}
&\frac{\delta(t-\|y-z\|)}{4\pi\|y-z\|}, &&N=3,\\
&\frac{H(t-\|y-z\|)}{2\pi \sqrt{ t^2-\|y-z\|^2 }}, &&N=2,
\end{aligned}
\right.
\en
where $\delta$ is the Dirac delta distribution and $H$ is the Heaviside step function. The frequency domain fundamental solution $G(y,z,k)=u^i(y;z,k)$ is the inverse Fourier transform of the time domain fundamental solution, i.e.,
\ben
G(y,z,k) := \int_\mathbb{R} e^{ikt} \hat{G}(y,t;z) \, dt=  \left\{
\begin{aligned}
& \frac{e^{ik\|y-z\|}}{4\pi\|y-z\|}, &N=3,\\
& \frac{i}{4} H_0^{(1)}(k\|y-z\|), &N=2.
\end{aligned}
\right.
\enn
Here $H_0^{(1)}$ denotes the Hankel function of the first kind and order zero.
The fundamental solution and its gradient satisfy the high-frequency asymptotic expansions \cite{CK}
\be \label{eq-Asym-Funda2}
\begin{aligned}
G(y,z,k) &= C_N(k)\frac{e^{ik\|y-z\|}}{\|y-z\|^{\frac{N-1}{2}}}\left( 1+O\left( \frac{1}{k} \right) \right), &k\to\infty,\\
\nabla_y G(y,z,k) &= ik \theta_z(y) C_N(k)\frac{e^{ik\|y-z\|}}{\|y-z\|^{\frac{N-1}{2}}}\left( 1+O\left( \frac{1}{k} \right) \right), &k\to\infty,
\end{aligned}
\en
where $C_N(k) = \frac{i}{2k}\left( \frac{k}{2\pi i} \right)^{(N-1)/2}$.

\subsection{Pseudo-differential operators}
Pseudo-differential operators (PDOs) are used to characterize the operator properties of "wavefront propagation and boundary interaction" in high-frequency scattering. Their principal symbols determine the behavior of the leading term of the scattered field, serving as a core tool for deriving asymptotic expansions.
Now we recall the definition of pseudo-differential operators (PDOs) on $\R^N$.
\begin{definition}
A function $a\in C^\infty(\R^N\times \R^N)$ is called a symbol in $S^m(\R^N)$ if, for all multi-index $\alpha, \beta$, there exists a constant  $C_{\alpha, \beta}$ such that
\ben
|\partial_\xi^\alpha \partial_x^\beta a(x, \xi)|\leq C_{\alpha, \beta}(1+\|\xi\|)^{m-|\alpha|}, \quad x,\xi\in \R^N,\quad m\in \R
\enn
It is called a classical symbol of order $m$, i.e., $a\in S^m_{cl}(\R^N)$, if  there exists a sequence of homogeneous symbols $a_{m-j}(x, \xi)$ of order $m-j$ with respect to $\xi$, i.e., 
\begin{equation*}
    a_{m-j}(x, \lambda \xi) = \lambda^{m-j} a_{m-j}(x, \xi), \quad\lambda \in \R, \quad j=0,1,\cdots
\end{equation*}
such that
$$
a - \sum_{j=0}^J a_{m-j}\in S^{m-J-1}(\R^N).
$$
\end{definition}

In particular, for a classical symbol $a$, we have
\be\label{eq:symbolerror}
|a(x, \xi)-a_m(x,\xi)|\leq C_{0,0} \|\xi\|^{m-1}, \quad \|\xi\|\to\infty
\en
for some constant $C_{0,0}$.

Each symbol corresponds to a unique pseudo-differential operator (PDO).
\begin{definition}
An operator $P$ is a classical PDO of order $m$ with symbol $a\in S^m_{cl}(\R^N)$ if
$$
Pu =  \int_{\R^N} e^{ix\cdot \xi} a(x, \xi)\hat{u}(\xi){d\xi},\quad u\in C_0^\infty(\R^N),
$$
where $\hat{u}$ is the Fourier transform of $u$, i.e.,
\begin{equation*}
    \hat{u}(\xi):=\frac{1}{(2\pi)^N}\int_{\R^N} e^{-i x\cdot\xi} u(x)dx.
\end{equation*}
The space of all classical PDOs of order m is denoted by $OPS^m_{cl}(\R^N)$.
\end{definition}

The definition of PDOs on $\R^N$ can be naturally extended to PDOs on a smooth manifold.
\begin{definition}[PDO on manifold]
    Let  $M$ be a smooth manifold of dimension $N$. We call an operator $P: C_0^\infty(M) \mapsto \mathcal{D}'(M)$ a classical PDO of order $m$ in $M$ (denoted $P\in OPS_{cl}^m(M)$) if its localized action is equivalent to the action of standard PDO on $\R^N$. Specifically, there exists an open cover $\{\Omega_j\}$ of $M$ and diffeomorphisms $\{F_j: \Omega_j \to \mathcal{O}_j \subset \R^N\}$ such that
\begin{equation*}
    (Pu)(y) = [Q_{ij}(\chi_{\Omega_j}(\cdot)u(F_j^{-1}(\cdot)))](F_i(y)), \quad u\in C_0^\infty(M), \quad y\in \Omega_i
\end{equation*}
for some $Q_{ij}\in OPS^m_{cl}(\R^N)$. Here $\chi_{\Omega_j}$ is the character function of $\Omega_j$.
\end{definition}


More generally, we consider the Fourier integral operator $P: C_0^\infty(\R^N) \mapsto \mathcal{D}'(\R^N)$ defined by
$$
Pu = \int_{\R^N} e^{i \phi(x, \xi)} a(x, \xi)\hat{u}(\xi)d\xi,
$$
where $\phi(x, \xi)$ is a phase function. In particular, we call $P$ a PDO of order $m$ in $\Omega$ if $\phi = x\cdot \xi$ in a domain $\Omega \subset \{(x, \xi)\in \R^N\times \R^N\}$.

\subsection{Time-frequency correspondence}

For the time domain scattering problem, we have the following boundary value problem:
\be \label{eq-scatteringEQs-timedomain}
\begin{aligned}
\partial^2_t w - \Delta w&=0, &\mbox{in } \R^N\backslash \ov{D},\\
\hat{\mathcal{B}} w &= -\hat{\mathcal{B}} w^i, & \mbox{on } \partial D,\\
w(x,t)&= 0, & t\ll 0,
\end{aligned}
\en
where $w^i = \hat{G}(x,t;z)$.
The time-domain boundary condition $\hat{\mathcal{B}}$ takes one of the following:
\be \label{eq-BC-timedomain}
(1) \hat{\mathcal{B}} w = w, \qquad
(2) \hat{\mathcal{B}} w = \frac{\partial w}{\partial n}, \qquad
(3) \hat{\mathcal{B}} w = \frac{\partial w}{\partial n} - \gamma \frac{\partial w}{\partial t}.
\en
It is straightforward to verify that $u(x, k):=\int_\mathbb{R} e^{ikt} w(x,t)dt$ solves the frequency scattering problem with the corresponding boundary condition and incident field $u^i = G(x,z,k)$.



\subsection{Proof of Lemma \ref{lemma:UniqueYplus}}
For convenience, we restate the lemma \ref{lemma:UniqueYplus} as follows.
\begin{lemma}\label{lem:geometric}
If $\partial D_{x,z}^+$ is not empty, there exists a unique stationary point $y^+\in \partial D_{x,z}^+$ of $\psi$, and this point $y^+$ is non-degenerate.
\end{lemma}
\begin{proof}
Let $y^+\in \partial D^+_{x,z}$ be a stationary point of $\psi(y) = \|x-y\|+\|y-z\|$, we define an elliptic disk (or solid ellipsoid)
$$
E^+:=\{y\in \R^N: \|y-x\|+\|y-z\| < \psi(y^+)\},
$$
which is tangent to $D$ at $y^+$ because the normal $n(y^+)$ is parallel to $\nabla\psi(y^+)$ and $\nabla\psi(y^+)$ is parallel to the normal of $\pa E^+$ at $y^+$.

Since $y^+\in \partial D^+_{x,z}$, we have $n(y^+)\cdot \nabla \psi(y^+)<0$. Hence $E^+$ and $D$ are separated by the hyperplane $n(y^+)\cdot (x-y^+) = 0$. Therefore, 
\begin{equation}\label{EcapD=emptyset}
    E^+\cap \ov{D}=\emptyset.
\end{equation}

{\bf Uniqueness.} We first prove the uniqueness by contradiction. Assume on the contrary, there are two different stationary points $y^+$ and $\tilde{y}^+$ on $\partial D_{x,z}^+$. Then we define two elliptic disks (or solid ellipsoids) $E^+, \tilde{E}^+$ as above. 

If $E^+= \tilde{E}^+$, then $(y^++\tilde{y}^+)/2\in E^+$, which is strictly convex, and $(y^++\tilde{y}^+)/2\in \ov{D}$. This contradicts the equality \eqref{EcapD=emptyset}.
Otherwise, suppose $\psi(y^+)>\psi(\tilde{y}^+)$, then $\tilde{y}^+\in E^+\cap \ov{D}$, a contradiction to \eqref{EcapD=emptyset}. This completes the proof of uniqueness.

{\bf Existence.} We next give a constructive proof of the existence. Let $y_0\in \pa D$ be such that $\psi(y_0)=\inf_{y\in \partial D} \psi(y)$, which is naturally a stationary point of $\psi$. It suffices to show that $y_0\in \partial D^+_{x,z}$. 
Denote by $\Delta_{xyz}$ the closed triangle formed by $x, z$ and some $y\in \partial D^+_{x,z}$. For any point $x_0\in \Delta_{xyz}$, there exists $\lambda, \mu\in[0,1]$ such that
\ben
x_0-y = \mu\left(\lambda (x-y) + (1-\lambda)(z-y)\right).
\enn
Hence, $(x_0-y)\cdot n(y)\geq 0$ because $(x-y)\cdot n(y)>0$ and $(z-y)\cdot n(y)>0$. Moreover, $(x_0-y)\cdot n(y)= 0$ if and only if $x_0=y$. Since $\ov{D}$ is convex, we have $(y_0-y)\cdot n(y)\leq 0$ for any $y_0\in \ov{D}$. It follows that $\Delta_{xyz} \cap \ov{D} = \{y\}$, which implies that $\{\lambda x+(1-\lambda)z|\lambda\in [0,1]\}\cap \ov{D}=\emptyset$. Consequently, 
\begin{equation*}
    \psi(y)>\|x-z\|\quad\mbox{for all }\, y\in \partial D.
\end{equation*}
In particular, $\psi(y_0) >\|x-z\|$. Following the arguments before the uniqueness proof,  we define an elliptic disk (or solid ellipsoid) with focuses $x, z$:
\ben
E_0:= \{y\in \R^2: \|y-x\|+\|y-z\| < \psi(y_0)\}.
\enn
Then $E_0\cap \ov{D} = \emptyset$. Hence $E_0$ and $D$ lie on the different side of the hyperplane $\{y\in \R^N|(y-y_0)\cdot n(y_0)=0\}$. It's straightforward to verify 
\begin{equation*}
    n(y_0)\cdot(y_0-x)<0\quad\mbox{and}\quad n(y_0)\cdot(y_0-z)<0,
\end{equation*}
i.e., $y_0\in \partial D^+_{x,z}$. This completes the proof of existence.

\textbf{Non-degeneracy.} We establish this by demonstrating that the surface Hessian $H_{\partial D}\psi(y_0)$ is strictly positive definite, which requires $v^T H_{\partial D}\psi(y_0) v > 0$ for any non-zero tangent vector $v \in T_{y_0}(\partial D)$. Let $\sigma(s)$ be the geodesic on $\partial D$ passing through $y_0$ with initial velocity $v$, i.e., $\sigma(0) = y_0$ and $\sigma'(0) = v$. Evaluating the second derivative of $\psi$ along $\sigma(s)$ at $s=0$ yields
\ben
v^T H_{\partial D}\psi(y_0) v = \left. \frac{d^2}{ds^2} \psi(\sigma(s)) \right|_{s=0} = v^T \nabla^2\psi(y_0)v + \nabla\psi(y_0)\cdot \sigma''(0).
\enn
Since $\sigma(s)$ is a geodesic on the boundary of the strictly convex obstacle $D$, its acceleration $\sigma''(0)$ at the stationary point is entirely normal and points inward (opposite to the outward unit normal $n(y_0)$). Furthermore, because $y_0$ lies in the illuminated region $\partial D^+_{x,z}$, we have $\nabla\psi(y_0) \cdot n(y_0) < 0$. Consequently, the geometric term is non-negative: $\nabla\psi(y_0)\cdot \sigma''(0) \geq 0$. 

For the space Hessian term of $\psi$, straightforward calculation gives
\ben
v^T \nabla^2\psi(y_0)v = \frac{\|v\|^2 - (v\cdot \theta_x(y_0))^2}{\|x-y_0\|} + \frac{\|v\|^2 - (v\cdot \theta_z(y_0))^2}{\|z-y_0\|} > 0,
\enn
where the strict inequality holds because the tangent vector $v$ cannot be parallel to the transversal incident and observation ray directions $\theta_x(y_0)$ and $\theta_z(y_0)$.

Combining these results, we conclude that $v^T H_{\partial D}\psi(y_0) v > 0$. Thus, $H_{\partial D}\psi(y_0)$ is strictly positive definite, implying $\det H_{\partial D}\psi(y_0) > 0$. Therefore, the stationary point $y_0 \in \partial D^+_{x,z}$ is non-degenerate.
\end{proof}

We remark that, for $x=z$, the stationary point of $\psi$ on $\partial D^+_x$ is exactly $y_0:=\mathop{arg min}_{y\in \partial D} \|y-x\|$, i.e., the closest point to $x$ on $\pa D$. 

\section{Proof of Main Results}\label{Proofs}

\subsection{Operators on boundary}
To reduce the wave propagation problem in the unbounded exterior domain $\R^N\setminus \overline{D}$ to an equivalent formulation on the boundary $\partial D$, we introduce two operators $\mathcal{F}_1$ and $\mathcal{F}_2: C^\infty(\partial D\times \mathbb{R}) \to \mathcal{E}'(\partial D\times \mathbb{R})$ associated with the time-domain scattering problem \eqref{eq-scatteringEQs-timedomain} as follows
\be \label{eq-F1F2}
\begin{aligned}
\mathcal{F}_1(\hat{\mathcal{B}} w) := w|_{\partial D\times \mathbb{R}},\quad
\mathcal{F}_2(\hat{\mathcal{B}} w) := \left.\frac{\partial w}{\partial n}\right|_{\partial D\times \mathbb{R}}.
\end{aligned}
\en
It is shown in \cite{Majda1976, Taylor1976} that  $\mathcal{F}_1$ and $\mathcal{F}_2$ are Fourier integral operators of finite order. Moreover, $\mathcal{F}_1$ and $\mathcal{F}_2$ are PDOs of finite order away from the grazing direction, and their principal symbols are given as follows:
\begin{proposition}[\cite{Majda1976}] \label{prop-F1F2}
For Dirichlet boundary condition, $\mathcal{F}_1=I$ and the principal symbol of $\mathcal{F}_2$ is
$$
i\tau^{\dag}\in S^1_{cl}(\partial D\times \mathbb{R})
$$
away from the grazing direction, i.e., $\|\xi\|<|\xi_0|$, where
\be\label{eq:tauplus-general}
\tau^{\dag}=\tau^{\dag}(y,t,\xi,\xi_0):=\left\{
\begin{aligned}
&-\sqrt{ \xi_0^2-\|\xi\|^2 }, &\xi_0>0,\\
&\sqrt{ \xi_0^2-\|\xi\|^2 }, &\xi_0<0.
\end{aligned}
\right.
\en
For Neumann and Robin boundary condition, the principal symbol of $\mathcal{F}_1, \mathcal{F}_2$ are, respectively,
$$
(i\tau^{\dag} - i\gamma \xi_0)^{-1}\in S^{-1}_{cl}(\partial D\times \mathbb{R})\quad\mbox{and}\quad i\tau^{\dag}(i\tau^{\dag} - i\gamma \xi_0)^{-1}\in S^{0}_{cl}(\partial D\times \mathbb{R})
$$
away from the grazing direction.
\end{proposition}

We make the following remarks on Proposition\ref{prop-F1F2}:
\begin{itemize}
    \item[1.] Majda \cite{Majda1976} does not state this proposition directly. Instead, it states the localized version in page 271 where $\partial D$ is split by partition of unity. Then Majda carefully excludes the influence of the partition function. Here we combine these results as Proposition \ref{prop-F1F2}.
    \item[2.] Majda \cite{Majda1976} proves Proposition \ref{prop-F1F2} for all cases except the Robin boundary condition in two dimensions, as the energy decay results of Morawetz \cite{Morawetz1975} does not hold in this case. We resolve this issue in \cite{LiLiuShi2025} via the potential operator approach.
\end{itemize}

\subsection{Lemma on PDO action}

\begin{lemma}\label{lem:pdo-action}
Given a $N$-dimension smooth manifold $M$ and two functions $\varphi, F\in C^\infty(M)$.
For a clasical PDO $P$ of order $m$ in $M$ with homogeneous principal symbol $a_m$, we have
$$
(P(e^{ik \varphi} F))(x) = k^m a_m(x, \nabla_x \varphi ) e^{ik \varphi(x)} F(x)+O(k^{m-1}), \quad k\to\infty.
$$
\end{lemma}
\begin{proof}
We prove for the case when $M=\R^N$. The general case can be proved similarly.

Let $a(x,\xi)$ be the symbol of $P$, with the help of \eqref{eq:symbolerror}, we obtain
\ben
\begin{aligned}
(P(e^{ik \varphi} F))(x) &= \left(\frac{1}{2\pi}\right)^N\int_{\R^N}\int_{\R^N} e^{i\xi\cdot(x-y) + ik \varphi(y)} F(y) a(x, \xi) dyd\xi\\
&=\left(\frac{k}{2\pi}\right)^N\int_{\R^N}\int_{\R^N} e^{ik\Phi(x,y,\eta)} F(y) a(x, k\eta) dyd\eta \\
&= \left(\frac{k}{2\pi}\right)^N k^m\int_{\R^N}\int_{\R^N} e^{ik\Phi(x,y,\eta)} F(y) \left(a_m(x, \eta) + O\left(\frac{1}{k}\right)\right)dyd\eta
\end{aligned}
\enn
with $\Phi(x,y,\eta) := \eta\cdot(x-y)+\varphi(y)$. 
Straightforward calculations show that 
\begin{equation}\label{gradient_of_Phi}
    \partial_y \Phi = \nabla_y \varphi -\eta, \partial_\eta \Phi = x-y
\end{equation}
and the Hessian matrix of $\Phi$ with respect to $(y,\eta)$ is given by
\ben
\nabla^2 \Phi=
\begin{pmatrix}
    \nabla_y^2\varphi & -I\\
    -I& 0
\end{pmatrix},
\enn
which is non-degenerate. From \eqref{gradient_of_Phi} we deduce that $\Phi$ has a non-degenerate stationary point $(x, \nabla_x\psi)$.
This lemma then follows by the stationary phase method (see \cite{Zworski2012}[Thm 3.16]).
\end{proof}


\subsection{Proof of Theorem \ref{thm:general} and Theorem \ref{thm:backscattering}}
Let $\left\langle\cdot, \cdot\right\rangle$ be the action of a distribution on a test function defined by
\ben
\left\langle T, u\right\rangle := \int_{\partial D}\int_\mathbb{R} T \ov{u} dtds(y),
\enn
where $\ov{u}$ is the complex conjugate of $u$.
By Green's formula, we derive the following representation for the scattered field:
\begin{align}\label{eq-doAdjoint}
u^s(x;z,k) &= \int_{\partial D}\left[u^s(y;z,k) \frac{\partial G(x,y,k)}{\partial n(y)} -
\frac{\partial u^s}{\partial n}(y;z,k) G(x,y,k)\right] ds(y) \notag\\
&=C_N(k) \int_{\partial D} \frac{e^{ik\|x-y\|} }{\|x-y\|^{\frac{N-1}{2}}}\left[ ik \theta_x(y)\cdot n(y) u^s(y;z,k) - \frac{\partial u^s}{\partial n}(y;z,k)  
\right] ds(y)\notag\\
&= C_N(k) \int_{\partial D}\int_\mathbb{R} \frac{e^{ik\|x-y\|} }{\|x-y\|^{\frac{N-1}{2}}}
\left[ ik \theta_x(y)\cdot n(y) w(y,t) - \frac{\partial w}{\partial n}(y,t)  
\right] dtds(y)\notag\\
&= -C_N(k)\left\langle(ik \theta_x(y)\cdot n(y) \mathcal{F}_1 - \mathcal{F}_2) \mathcal{B} w^i,\, \frac{e^{ik\varphi(y,t)} }{\|x-y\|^{\frac{N-1}{2}}}
\right\rangle\notag\\
&= -C_N(k)\left\langle
\mathcal{B} w^i,\, (-ik  \mathcal{F}_1^*\theta_x(y)\cdot n(y) - \mathcal{F}_2^*)\left( \frac{e^{ik\varphi(y,t)} }{\|x-y\|^{\frac{N-1}{2}}}
\right)\right\rangle.
\end{align}
Here, $\varphi(y,t)=-\|x-y\|-t$ is the phase function, $w$ is the solution of the boundary value problem \eqref{eq-scatteringEQs-timedomain} and $\mathcal{F}_1^*$, $\mathcal{F}_2^*$ are the adjoint operators of $\mathcal{F}_1, \mathcal{F}_2$, respectively.

By \eqref{eq:tauplus-general}, 
\ben
\tau^{\dag}= \tau^{\dag}(y,t,\Grad_y \varphi(y,t), \partial_t\varphi(y,t))=\tau^{\dag}(y,t,-P_{\partial D}(y)\theta_x(y),-1)  = |\theta_x(y)\cdot n(y)|.
\enn
Hence the grazing directions are characterized by $\theta_x(y)\cdot n(y)=0$. By Lemma \ref{lem:pdo-action} and Propsition \ref{prop-F1F2}, when $\theta_x(y)\cdot n(y)\neq 0$, we have
\ben
\begin{aligned}
\quad(-ik  \mathcal{F}_1^*\theta_x(y)\cdot n(y) - \mathcal{F}_2^*)\left(
\frac{e^{-ik\|x-y\|-t} }{\|x-y\|^{\frac{N-1}{2}}} \right) = C_{\mathcal{B}}(y;x,\gamma) \frac{e^{-ik\|x-y\|-t} }{\|x-y\|^{\frac{N-1}{2}}} + O\left( k^\alpha\right),\quad k\to\infty,
\end{aligned}
\enn
where 
\ben
C_{\mathcal{B}}(y;x,\gamma) = \left\{
\begin{aligned}
&-ik\left(\theta_x(y)\cdot n(y)-\tau^{\dag}\right), &\text{for Dirichlet case;}\\
&\frac{\theta_x(y)\cdot n(y)-\tau^{\dag}}{\tau^{\dag} +\gamma(y)}, &\text{for Neumann or Robin case}
\end{aligned}
\right.
\enn
and 
\ben
\alpha = \left\{
\begin{aligned}
&0, &\text{for Dirichlet case;}\\
&-1, &\text{for Neumann or Robin case.}
\end{aligned}
\right.
\enn
Hence, the quantity $C_{\mathcal{B}}$ vanishes on the non-illuminated side (where $\theta_x\cdot n>0$). The leading term of \eqref{eq-doAdjoint} is determined by the wave behavior on the illuminated side or the grazing part.

Let $\Gamma_G\subset \partial D$ denote the neighborhood of the grazing points where $\theta_x(y)\cdot n(y)=0$. We then have
\begin{align}\label{eq-doPartition}
u^s(x;z,k) &= -C_N(k) \int_{\partial D \setminus \Gamma_G} \left[
\overline{C_{\mathcal{B}}(y;x,\gamma) }\frac{e^{ik\|x-y\|}}{\|x-y\|^{\frac{N-1}{2}}} +O(k^{\alpha})\right]
\int_\mathbb{R} e^{ikt} \hat{\mathcal{B}}w^i dtds(y) +I_{\Gamma_G} \notag\\
&=-C_N(k) \int_{\partial D\setminus \Gamma_G} \left[
\overline{C_{\mathcal{B}}(y;x,\gamma) }\frac{e^{ik\|x-y\|}}{\|x-y\|^{\frac{N-1}{2}}} +O(k^{\alpha})\right]\mathcal{B} u^i ds(y)+I_{\Gamma_G},
\end{align}
where $I_{\Gamma_G}$ is defined as
\be \label{eq:I_GammaG}
I_{\Gamma_G} := C_N(k) \int_{\Gamma_G} \int_\R \frac{e^{ik\|x-y\|+t} }{\|x-y\|^{\frac{N-1}{2}}} \left[ ik \theta_x(y)\cdot n(y) w(y,t) - \frac{\partial w}{\partial n}(y,t)  
\right] dt ds(y).
\en
Since $\mathcal{F}_1$ and $\mathcal{F}_2$ are Fourier integral operators, they constrain the wavefront sets $\mathop{WF}(w)$ and $\mathop{WF}(\frac{\partial w}{\partial n})$ via the singular support of $w^i|_{\partial D\times \R}$. More precisely, we have
\be \label{eq-WFofw0}
\begin{aligned}
    &\mathop{WF}(w|_{\partial D\times \R}),\quad \mathop{WF}\Big( \frac{\partial w}{\partial n}\Big{|}_{\partial D\times \R} \Big) \\
    & \subset \{
    (y,t,\xi,\xi_0)\in (\partial D \times \R)\times (\R^{N-1}\times \R):\, \|y-z\|-t=0, (\xi,\xi_0) = (P_{\partial D}(y)\theta_z(y),-1)
    \}\\
\end{aligned}
\en
Consider the phase function $\tilde{\varphi}:=\|y-x\|+t$ in \eqref{eq:I_GammaG}, it shows that 
\ben
\left(\Grad_y\tilde{\varphi},\frac{\partial \tilde{\varphi}}{\partial t}\right)=(P_{\partial D}(y)\theta_x(y),1).
\enn
This direction is parallel to the direction in \eqref{eq-WFofw0} only when 
\be\label{eq:WFcondi}
\theta_x(y)+\theta_z(y) \mathop{//} n(y).
\en
Geometrically, this condition is exactly Snell's law of specular reflection, which is satisfied exclusively at the stationary points $y_0\in \partial D$. Since $y_0$ is away from the grazing neighborhood $\Gamma_G$, the condition \eqref{eq:WFcondi} does not hold for any $y\in \Gamma_G$. By the fundamental properties of the wavefront set, $I_{\Gamma_G}$ decays rapidly as $k\to \infty$ \cite{LH-SPM}, i.e.,
\ben
I_{\Gamma_G} = O(k^{-p}) \quad \mbox{for any }p>0.
\enn

Now we compute the non-grazing part of \eqref{eq-doPartition}. Direct calculations yield
\ben 
\begin{aligned}
\mathcal{B}u^i &= \left\{
\begin{aligned}
&G(y,z,k), &\text{for Dirichlet;}\\
&\frac{\partial G(y,z,k)}{\partial n(y)} + ik \gamma G(y,z,k), &\text{for Neumann or Robin;}
\end{aligned}
\right.\\
&=\left\{
\begin{aligned}
&C_N(k)\frac{e^{ik\|y-z\|}}{\|y-z\|^{\frac{N-1}{2}}}\left( 1+O\left( \frac{1}{k} \right) \right), &\text{for Dirichlet;}\\
&ik(\theta_z\cdot n+\gamma)C_N(k)\frac{e^{ik\|y-z\|}}{\|y-z\|^{\frac{N-1}{2}}}\left( 1+O\left( \frac{1}{k} \right) \right), &\text{for Neumann or Robin;}
\end{aligned}
\right.
\end{aligned}
\enn


Under Assumption \ref{assumption:geometric}, we deduce by the stationary phase method that
\begin{align}\label{eq-Asym-u-proof}
&u^s(x;z,k) \notag\\
&= -ik C_{N}^2(k)\int_{\partial D\setminus \Gamma_G} \frac{e^{ik(\|x-y\|+\|y-z\|)}}{(\|x-y\|\|y-z\|)^{\frac{N-1}{2}}} (\theta_x(y)\cdot n(y)- \tau^{\dag})\frac{\gamma(y)+\theta_z(y)\cdot n(y)}{\gamma(y)+\tau^{\dag}} ds(y)+I_{\Gamma_G} \notag\\
&=A_N(x,z,k) \theta_x(y^+)\cdot n(y^+) \frac{\gamma(y^+)+\theta_z(y^+)\cdot n(y^+)}{\gamma(y^+)-\theta_z(y^+)\cdot n(y^+)}e^{ik\psi(y^+)} +O\left( k^{\frac{N-5}{2}} \right).
\end{align}
Here $y^+$ is the unique stationary point of $\psi$ in $\partial D^+_{x,z}$ as in Lemma \ref{lem:geometric}.

For Theorem \ref{thm:backscattering} (the backscattering case with $x=z$), the stationary point $y^+$ is the closest point to $x$ on $\partial D$. The above representation thus simplifies to
\be \label{eq-Asym-u0-backsca-proof}
u^s(x;x,k) = -A_N(x,x,k) \frac{\gamma(y^+)-1}{\gamma(y^+)+1}e^{2ik\|x-y^+\|}+O\left( k^{\frac{N-5}{2}} \right).
\en

\subsection{Proof of Theorem \ref{thm:uniqueness}}
It suffices to prove uniqueness in the case when $K = \{k_j=k_0+j\delta_k:\, j = 0, 1, \cdots\}$ with $\delta_k<\frac{\pi}{4R}$. Under the given assumptions, the asymptotic expansion \eqref{eq-Asym-u0-backsca-proof}  holds for all $x\in \Gamma_R$.
For any fixed $x\in \Gamma_R$, we define test functions $f_{x,k_j}$ of $t\in (0, 2R)$ by
\ben
\begin{aligned}
    f_{x, k_j}(t) &:=  \operatorname{Im} \left\{ 2k_j^{\frac{3-N}{2}} e^{\frac{N-3}{4}\pi i - 2ik_jt} u^s(x; x, k_j) \right\} \\
    &=\left(\frac{1}{2\pi\|x-y^+\|^2|\det H_{\partial D}\psi(y^+)|}\right)^{\frac{N-1}{2}} \frac{\gamma^+-1}{\gamma^++1}\sin(2k_j(t-\|x-y^+\|)) +O\left(\frac{1}{k_j}\right).
\end{aligned}
\enn
When $t = \|x-y^+\|$, we have $\lim_{j\to\infty} f_{x,k_j}(t) = 0$. For $t \neq \|x-y^+\|$, we claim that $\lim_{j\to\infty} |f_{x,k_j}(t)| > 0$ provided $\gamma^+\neq 1$. Thus, $\|x-y^+\|$ is uniquely determined for all $y^+$ with $\gamma^+\neq 1$.

We prove this claim by contradiction. Suppose, for contradiction, that $\lim_{j\to\infty} f_{x,k_j}(t) = 0$, and define $s:=t-\|x-y^+\|\in (-2R, 0)\cup (0, 2R)$. It follows that $\lim_{j\to\infty} \sin(2k_j s) =0$. Furthermore,
\ben
\lim_{j\to\infty} |f_{x,k_{j+1}}(t) - f_{x,k_{j}}(t)| =0,
\enn
which implies
\ben
|\sin(2k_{j+1}s)-\sin(2k_{j}s)| = 2\left|\cos\left((k_{j+1}+k_{j}) s\right) \sin \left(\delta_k s\right) \right|\to 0
\enn
as $j\to\infty$. Note that $\delta_k s \in (-\pi/2, 0)\cup (0, \pi/2)$, i.e., $\sin (\delta_k s)\neq 0$. Hence
\ben
\cos\Big((k_{j+1}+k_{j}) s\Big) = \cos(2 k_j s) \cos(\delta_k s) - \sin(2 k_j s) \sin(\delta_k s)\to 0.
\enn
Since $\sin(2k_j s)\to 0$, it follows that $|\cos(2k_j s)|\to 1$, which forces $\cos(\delta_k s)=0$. This contradicts the fact that $\delta_k s \in (-\pi/2, 0)\cup (0, \pi/2)$, which completes the proof of the claim.

We now turn to proving that the domain $D$ itself is uniquely determined. For any $x\in \Gamma_R$ with $\gamma(y^+(x))\neq 1$, we define the ball $B_x:= \{z\in B_R : \|z-x\| < \|x-y^+\|\}$. We claim that 
\ben
\overline{D} = \{z\in \R^N : \|z\|<R\}\backslash \bigcup_{x\in \Gamma_R} B_x.
\enn
Note that the union $\bigcup_{x\in \Gamma_R} B_x$ is fully determined from our data, since by Assumption \ref{assumption:boundcondi}, the equality $\gamma(x)=1$ holds for at most finitely many $x$.

Since $D$ is convex, the inclusion $\overline{D} \subset\{z\in \R^N : \|z\|<R\}\setminus \bigcup_{x\in \Gamma_R} B_x$ is straightforward. We prove the reverse inclusion by contradiction. Suppose, for contradiction, that there exists $z\in \R^N$ such that $\|z\|<R$, $z\notin \overline{D}$ and $z\notin B_x$ for all $x\in \Gamma_R$. Let $y_z\in\pa D$ be the projection of $z$ onto $\partial D$, i.e., $n(y_z)=(z-y_z)/\|z-y_z\|$. The projection property implies that
\ben
(z-y_z)\cdot(y-y_z)\leq 0, \quad \forall y\in D.
\enn
The ray starting at $y_z$ and pointing to $z$ intersects $\Gamma_R$ at some point $x_R$. Since $x_R-y_z = \lambda (z-y_z)$ for some $\lambda>1$, we have
\ben
(x_R-y_z)\cdot(y-y_z)\leq 0, \quad \forall y\in D,
\enn
i.e., $y_z$ is the projection of $x_R$ onto $\pa D$. This implies $\|z-x_R\|=(1-\lambda^{-1})\|y_z-x_R\|<\|y_z-x_R\|= \|x_R-y^+(x_R)\|$, contradicting the assumption that $z\notin B_{x_R}$. Consequently, $D$ is uniquely determined.

With $D$ now uniquely determined, the quantity
\ben
q:=\frac{\gamma^+-1}{\gamma^++1} = -\lim_{j\to\infty}  A^{-1}_N(x,x,k_j)e^{- 2ik_j\|x-y^+\|} u^s(x; x, k_j)
\enn
is uniquely determined. Consequently, the impedance value \begin{equation*}
    \gamma^+ = \frac{1+q}{1-q}
\end{equation*}
is then uniquely determined. Note that $q \equiv 1$ corresponds to the Dirichlet boundary condition, and $q\equiv -1$ corresponds to the Neumann boundary condition. This completes the proof of Theorem 2.6.

\section{Three-step reconstruction algorithm}\label{sec-Numeric}

This section is devoted to the numerical implementation of the inverse problem based on the asymptotic formulas derived in the previous sections. To stably decouple the strong nonlinearity between the obstacle's geometry and its physical properties, the algorithm proceeds in three steps: 
\begin{itemize}
    \item[(1)]{\bf Qualitative sampling method for shape reconstruction}, which extracts a robust initial guess of the boundary from the near-field backscattering data without requiring prior knowledge of the boundary condition;
    \item[(2)] {\bf Quantitative optimization for extracting boundary}, which refines the initial geometric guess by fitting a smooth, parameterized curve to the extracted reference points;
    \item[(3)] {\bf Decoupled boundary condition reconstruction}, which leverages the optimized geometry to identify and recover the physical properties.
\end{itemize}

\begin{algorithm}[ht]
\caption{Three-Step Algorithm for Identifying Obstacles from backscattering data}
\KwIn{Scattered field data $\{u^s(x_i, x_i, k_j)\}$ for source positions $\{x_i\}_{i=1}^{N_s}$ and wavenumbers $\{k_j\}_{j=1}^n$, imaging grid $\mathcal{G}$.}
\KwOut{Reconstructed boundary shape $\partial D$ and impedance function $\gamma$ on $\partial D$.}
\BlankLine
\textbf{(1) Qualitative shape reconstruction via the direct sampling method}\;
For grid point $z \in \mathcal{G}$, compute the total indicator $I_{\text{total}}(z)$ via \eqref{eq-total-indicator}\;
\BlankLine
\textbf{(2) Quantitative boundary refinement via shape optimization}\;
Extract boundary reference points $\{p_i\}_{i=1}^{N_r}$ using the thresholding rule defined in \eqref{eq-threshold}\;
Optimize the Fourier coefficients (or spherical expansion coefficients) by minimizing the regularized objective function $L$ in \eqref{eq-objective}\;
\BlankLine
\textbf{(3) Decoupled reconstruction of the boundary conditions}\;
\For{each source position $x_i$, $i = 1, \ldots, N_s$}{
    Find the closest boundary point $y_i^+ = \arg\min_{y \in \partial D} \|x_i - y\|$ on the optimized shape\;
    Compute the impedance ratio $q_i$ using the practical discrete sum in \eqref{eq-ratio-practical}\;
    Recover the discrete impedance value $\gamma(y_i^+)$ via \eqref{eq-recover-impedance}\;
}
Approximate the continuous global impedance function $\gamma$ by fitting the discrete pointwise values $\{\gamma(y_i^+)\}_{i=1}^{N_s}$ using a truncated Legendre polynomial series\;
\end{algorithm}

\subsection{Qualitative shape reconstruction via the direct sampling method}

Based on the asymptotic formula \eqref{eq-Asym-u0-backsca}, we define a direct sampling indicator function for each source position $x$ and imaging point $z$:
\be \label{eq-indicator}
I_x(z) = \left|\sum_{j=1}^{n} k_j^{1/2} e^{-2ik_j\|z-x\|} u^s(x;x,k_j)\right|,
\en
where $\{k_j\}_{j=1}^n$ are the discrete wavenumbers used in the measurements. The indicator function $I_x(z)$ is expected to attain a local maximum when $z$ lies on $\partial D$, specifically near the closest point $y^+$ to $x$. 

For multiple source positions $\{x_i\}_{i=1}^{N_s}$, we compute the indicator function for each source individually and then aggregate them to form a global indicator:
\be \label{eq-total-indicator}
I_{\text{total}}(z) = \sum_{i=1}^{N_s} \frac{I_{x_i}(z)}{\max_z I_{x_i}(z)}.
\en
This local normalization ensures each source contributes equally to the total indicator function by eliminating amplitude variations caused by the unknown impedance function $\gamma$.

\subsection{Quantitative boundary refinement via shape optimization}

We first extract candidate boundary points from the continuous indicator function landscape via thresholding. Specifically, we collect all grid points $z$ satisfying
\be \label{eq-threshold}
I_{\text{total}}(z) > \rho \cdot \max_z I_{\text{total}}(z),
\en
where $\rho \in (0,1)$ is a threshold parameter. These points serve as the geometric reference points $\{p_i\}_{i=1}^{N_r}$ for the subsequent shape refinement.
To ensure a smooth and stable reconstruction, the boundary $\partial D$ is parameterized via a truncated Fourier series in polar coordinates (or truncated spherical expansion in three dimension):
\be \label{eq-fourier-param}
r(\theta) = a_0 + \sum_{m=1}^{M} \left[a_m \cos(m\theta) + b_m \sin(m\theta)\right], \quad \theta \in [0, 2\pi),
\en
where $M$ denotes the truncation order. The shape optimization is then formulated as finding the optimal Fourier coefficients $\{a_m, b_m\}_{m=0}^M$ that minimize the discrepancy between the parameterized curve and the extracted reference points. The objective function is defined as:
\be \label{eq-objective}
L(\{a_m, b_m\}) = \frac{1}{N_r}\sum_{i=1}^{N_r} w_i d^2(p_i, \partial D), 
\en
where $d(p_i, \partial D)$ is the Euclidean distance from the reference point $p_i$ to the boundary  $\partial D$. 

The asymmetric spatial weights ($w_i$) in the objective function \eqref{eq-objective} are explicitly designed to ensure improvement of the accuracy and physical viability of the reconstructed shape. For a convex obstacle, the indicator function $I_x(z)$ is expected to be larger in the exterior of the obstacle than in the interior because two highlighting tangent lines will always intersect at the exterior of the obstacle. Consequently, a significantly larger penalty weight $w_i$ is assigned to interior points compared to exterior points. 

The optimization problem is solved iteratively via gradient descent.
The initial guess for the optimization is naturally  a simple circle with radius $a_0 = \frac{1}{N_r}\sum_{i=1}^{N_r} |p_i|$ and all higher-order coefficients initialized to zero. 

\subsection{Decoupled reconstruction of the boundary conditions}

Once the boundary geometry is precisely determined and fixed, the problem of recovering the physical properties (i.e., the impedance function $\gamma$) becomes a decoupled, well-posed procedure. For each source position $x_i$, we identify the closest point $y_i^+$ on the optimized boundary $\partial D$.

Based on the high-frequency near-field asymptotic expansion \eqref{eq-Asym-u0-backsca}, the scattered field behaves as:
\be \label{eq-impedance-formula}
u^s(x_i, x_i, k) \approx -\frac{e^{\frac{3-N}{4}\pi i}}{2k} \left( \frac{k}{2\pi\|x_i-y_i^+\|^2|\det H_{\partial D}\psi(y^+)|} \right)^{\frac{N-1}{2}} \frac{\gamma(y_i^+)-1}{\gamma(y_i^+)+1}e^{2ik\|x_i-y_i^+\|}.
\en

By isolating the impedance-dependent term, we define the impedance ratio $q_i = \frac{\gamma(y_i^+)-1}{\gamma(y_i^+)+1}$. Theoretically, $q_i$ can be extracted by taking the high-frequency limit:
\be \label{eq-extract-ratio}
q_i = \lim_{j\to\infty} k_j^{\frac{3-N}{2}} e^{\frac{3-N}{4}\pi i - 2ik_j\|x_i-y_i^+\|} u^s(x_i, x_i, k_j) \cdot C_i^{-1},
\en
where $C_i = -\frac{e^{\frac{3-N}{4}\pi i}}{2} \left( \frac{1}{2\pi\|x_i-y_i^+\|^2|\det H_{\partial D}\psi(y^+)|} \right)^{\frac{N-1}{2}}$ is a purely geometric factor that is now completely known from the previous step.

In numerical practice, since the available wavenumbers are discrete and finite, we approximate the limit by averaging over the multi-frequency measurements to suppress noise:
\be \label{eq-ratio-practical}
q_i \approx \frac{1}{n}\sum_{j=1}^{n} k_j^{\frac{3-N}{2}} e^{\frac{3-N}{4}\pi i - 2ik_j\|x_i-y_i^+\|} u^s(x_i, x_i, k_j) \cdot C_i^{-1}.
\en
Then the boundary condition is determined by the impedance ratio $q_i$:
\begin{itemize}
    \item If $q_i \approx 1$, then $\gamma(y_i^+) \to \infty$, corresponding to a \textbf{Dirichlet} boundary condition;
    \item If $q_i \approx -1$, then $\gamma(y_i^+) \to 0$, corresponding to a \textbf{Neumann} boundary condition;
    \item For intermediate values ($-1<q_i<1$), $\gamma(y_i^+)$ accurately yields the varying parameter of a \textbf{Robin} (impedance) boundary condition.
\end{itemize}

The point-wise impedance value is explicitly recovered via algebraic inversion:
\be \label{eq-recover-impedance}
\gamma(y_i^+) = \frac{1+q_i}{1-q_i}.
\en
Finally, the continuous global impedance function $\gamma$ is approximated by fitting the discrete pointwise values $\{\gamma(y_i^+)\}_{i=1}^{N_s}$ using a truncated Legendre polynomial series in the $L^2$-sense.

\section{Numerical examples}\label{NumExamples}

In this section, we present comprehensive numerical examples to demonstrate the effectiveness, robustness, and decoupled nature of the proposed three-step algorithm. All examples are computed in two dimensions ($N=2$). 

The scattered field data are generated by solving the forward problem using the Nyström method \cite{Kress95} with boundary integral equations. We simulate the near-field backscattering measurements $u^s(x_i, x_i, k_j)$ using $N_s = 32$ uniformly distributed point sources located on a measurement circle of radius $R=5$. The multi-frequency data are collected in the wavenumber range $k \in [10, 30]$ with a uniform step size $\delta_k = 0.2$.

To rigorously validate the robustness of the algorithm against measurement noise, we corrupt the scattered field data with $10\%$ relative complex Gaussian noise. Specifically, each measurement $u^s(x_i, x_i, k_j)$ is perturbed as follows:
\ben
u^\delta(x_i, x_i, k_j) = u^s(x_i, x_i, k_j) + 0.10 \frac{\xi_{i,j} + \mathrm{i}\zeta_{i,j}}{\sqrt{2}} |u^s(x_i, x_i, k_j)|,
\enn
where $\xi_{i,j}$ and $\zeta_{i,j}$ are independent standard normal random variables.

We test the algorithm on a non-trivial convex obstacle, specifically an ``egg-shaped'' domain parameterized as
\be \label{eq:eggShape}
\partial D = \left\{ \left( 1.5\cos(t), \frac{\sin(t)}{1+0.2\cos(t)} \right) |\, t \in [0, 2\pi) \right\}.
\en
To verify the decoupling capability of the algorithm, we consider four distinct physical boundary conditions:
(1) Dirichlet boundary condition;
(2) Neumann boundary condition;
(3) Constant Robin boundary condition with $\gamma(t) \equiv 0.5$;
(4) Variable Robin boundary condition with $\gamma(t) = 3 + \sin(t)$.
In the shape optimization step (Step 2), we set the threshold parameter to $\rho = 0.6$ and the asymmetric spatial weight to $w_i = 100$ for the interior points and $w_i = 1$ for the exterior points.

\subsection{Robust shape reconstruction: Step 1\&2}

We first evaluate the geometric reconstruction phase, which includes Step 1 (Shape Initialization via Direct Sampling) and Step 2 (Shape Optimization). The goal of these steps is to extract the obstacle's boundary without any prior knowledge of its physical boundary conditions.

Figure~\ref{fig:shape_reconstruction} illustrates the sequential geometric reconstruction procedure for all four boundary conditions. Despite the $10\%$ measurement noise, the indicator function robustly highlights the boundary region. The asymmetric distribution of the reference points stems from the weighted objective function \eqref{eq-objective}. Finally, as clearly shown in Figure~\ref{fig:shape_reconstruction}, the reconstructed shapes (dashed blue lines) are in excellent agreement with the true boundaries (solid red lines) across all different boundary conditions. This confirms that our geometric reconstruction steps are remarkably independent of the underlying physical parameters.

\begin{figure}[htbp]
\centering
\includegraphics[width=0.22\textwidth]{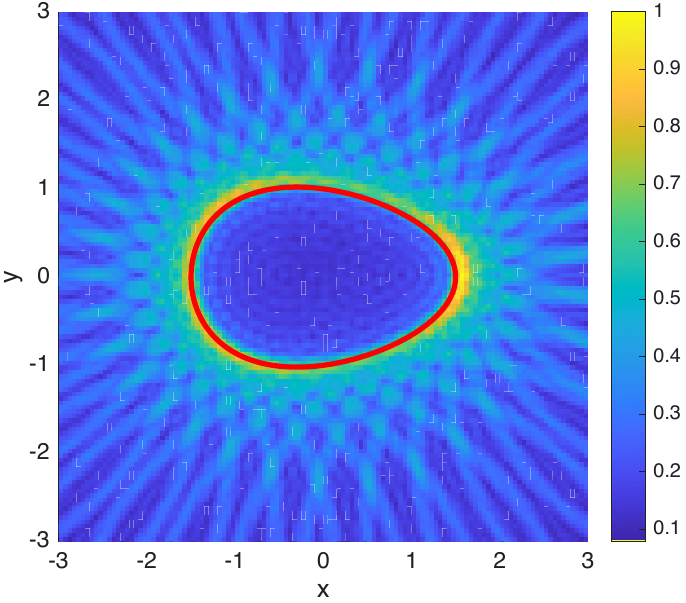}
\includegraphics[width=0.22\textwidth]{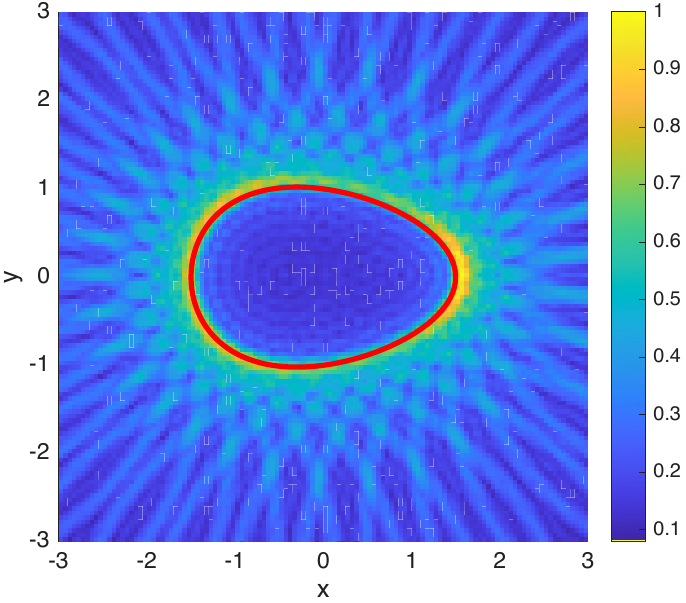}
\includegraphics[width=0.22\textwidth]{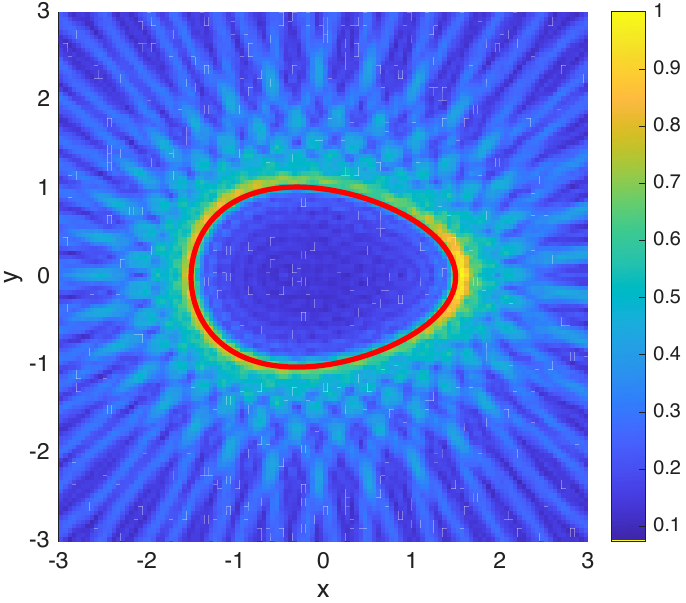}
\includegraphics[width=0.22\textwidth]{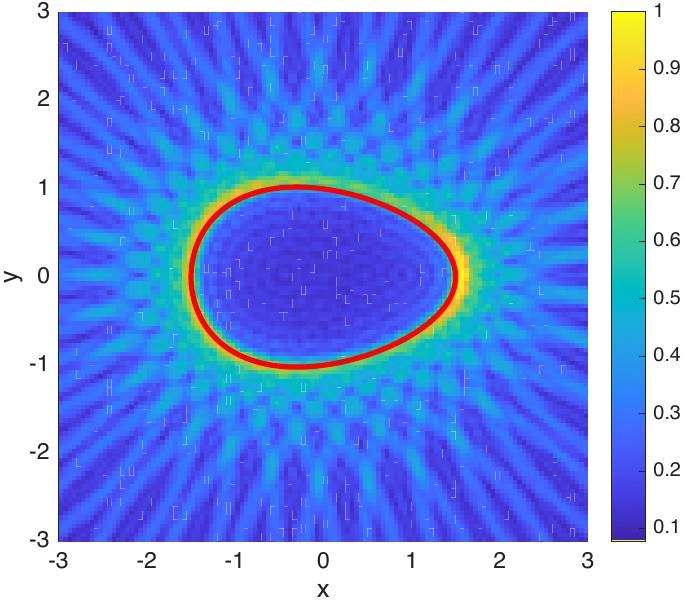}
\includegraphics[width=0.22\textwidth]{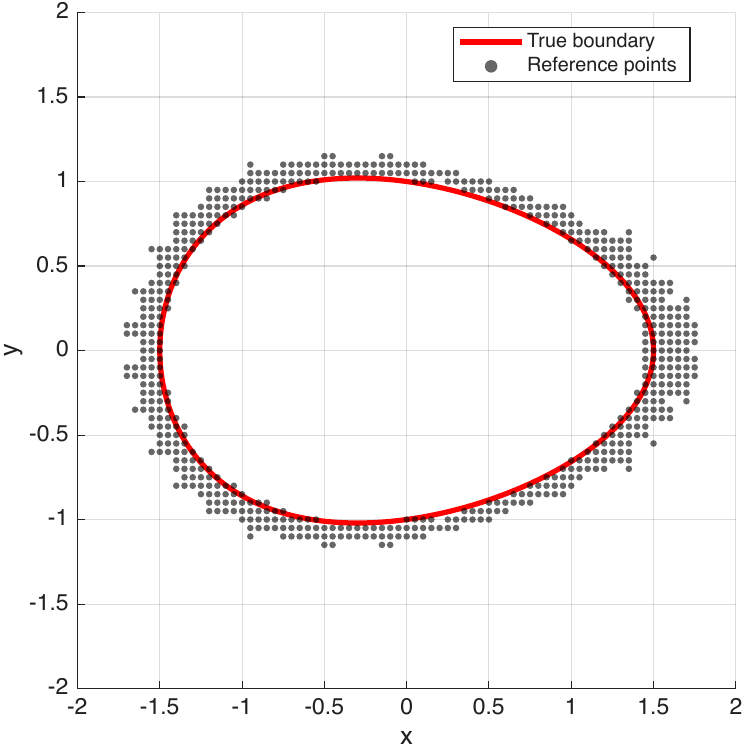}
\includegraphics[width=0.22\textwidth]{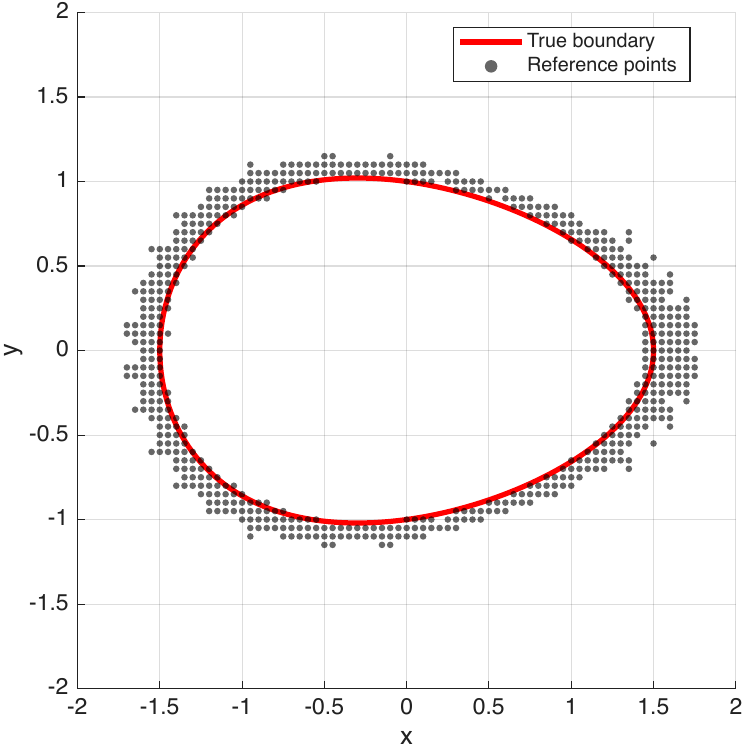}
\includegraphics[width=0.22\textwidth]{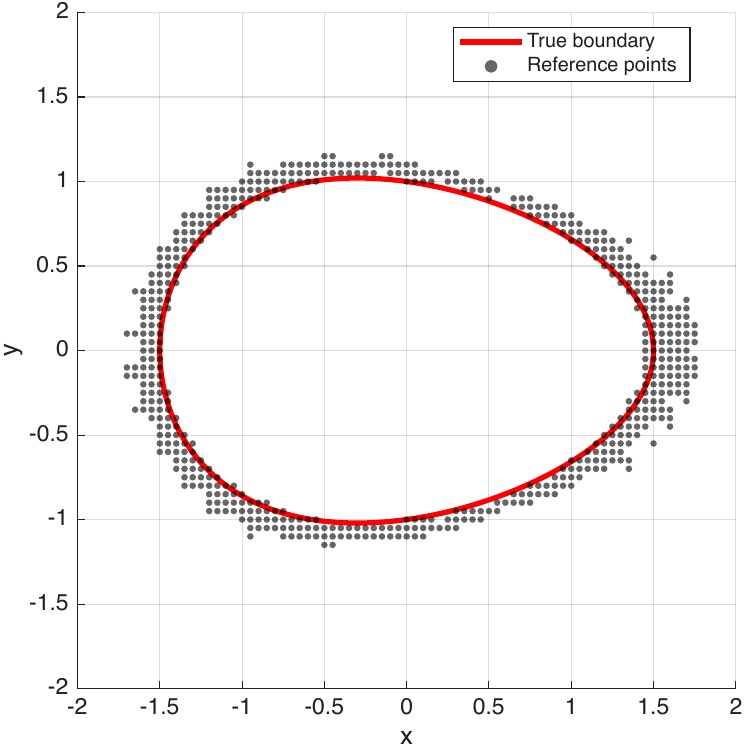}
\includegraphics[width=0.22\textwidth]{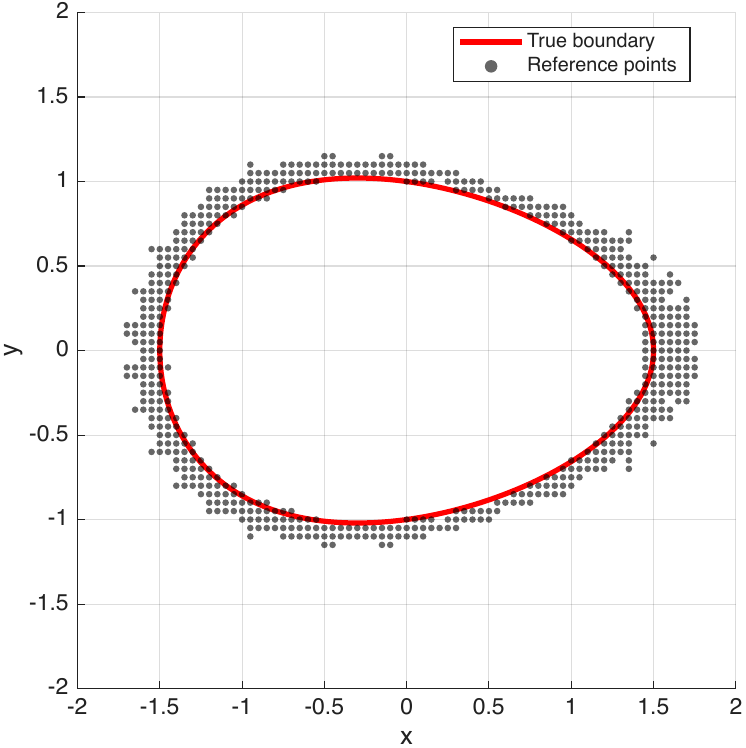}
\includegraphics[width=0.22\textwidth]{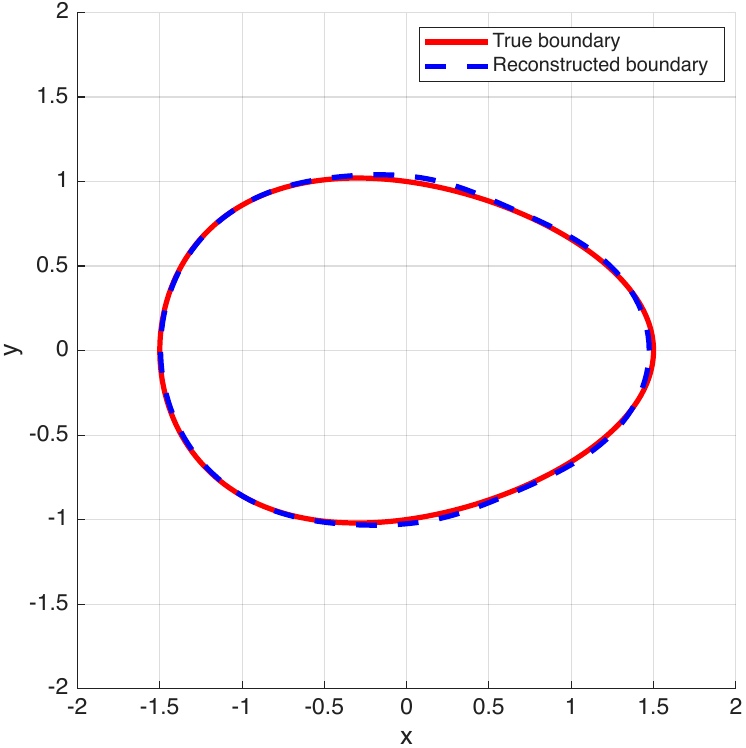}
\includegraphics[width=0.22\textwidth]{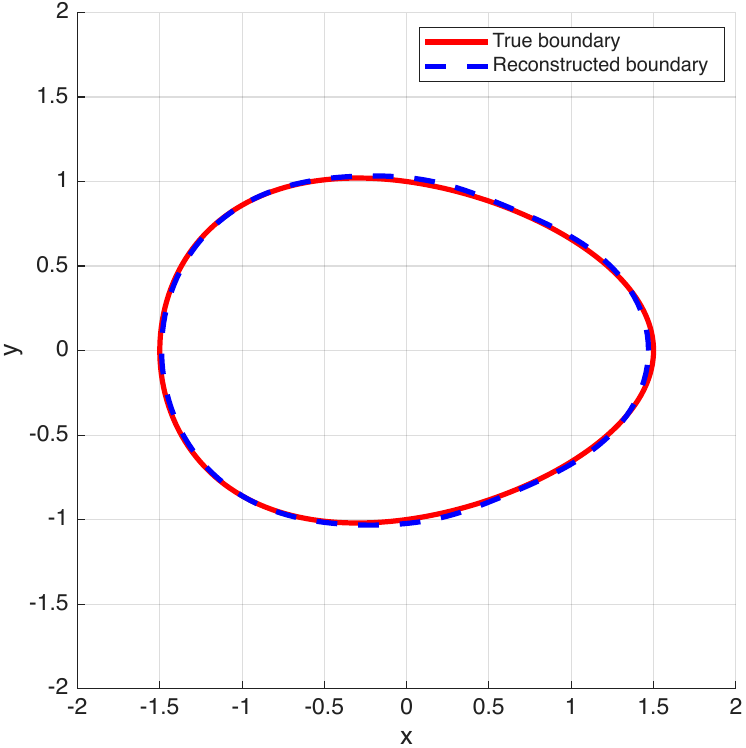}
\includegraphics[width=0.22\textwidth]{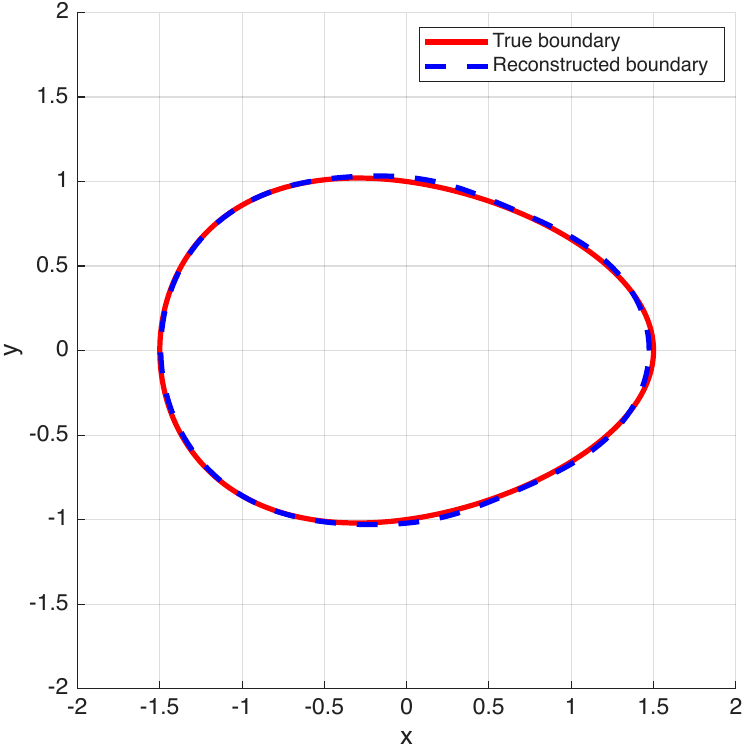}
\includegraphics[width=0.22\textwidth]{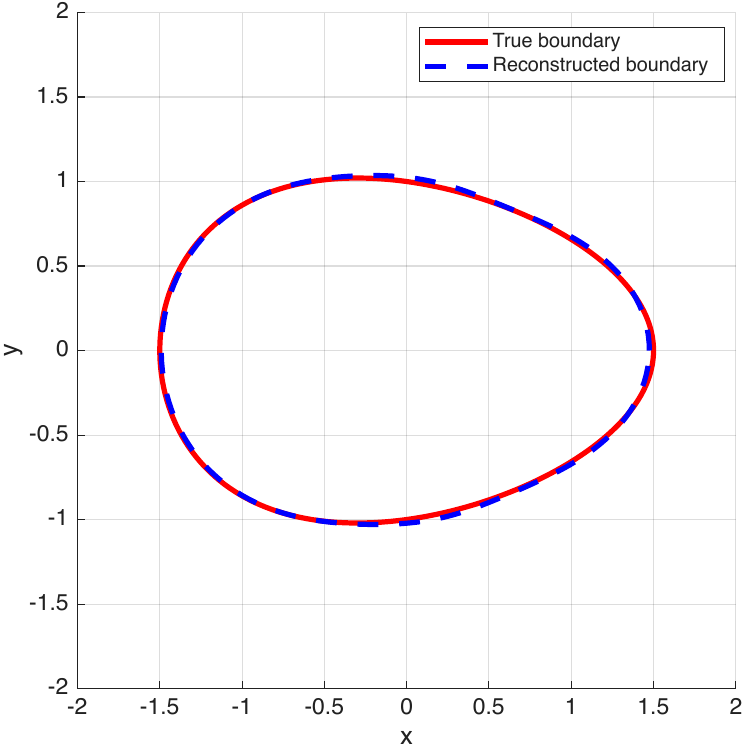}
\caption{Reconstruction of the shape. Left to right, columns correspond to the Dirichlet, Neumann, constant Robin, and variable Robin boundary conditions, respectively; top to bottom, rows represent the direct sampling indicator, extracted reference points, and the final optimized boundary, respectively; The true boundary is plotted as a solid red line.}
\label{fig:shape_reconstruction}
\end{figure}

\begin{figure}[htbp]
\centering
\includegraphics[width=0.22\textwidth]{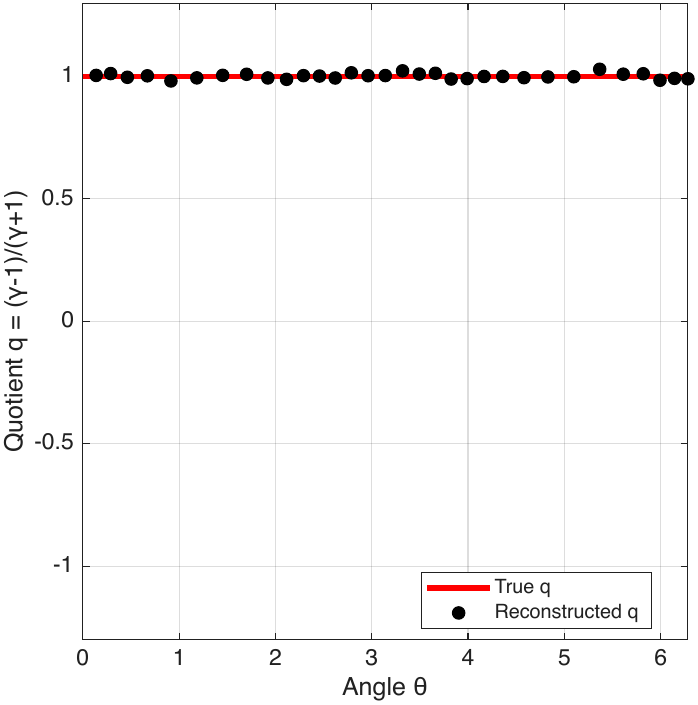}
\includegraphics[width=0.22\textwidth]{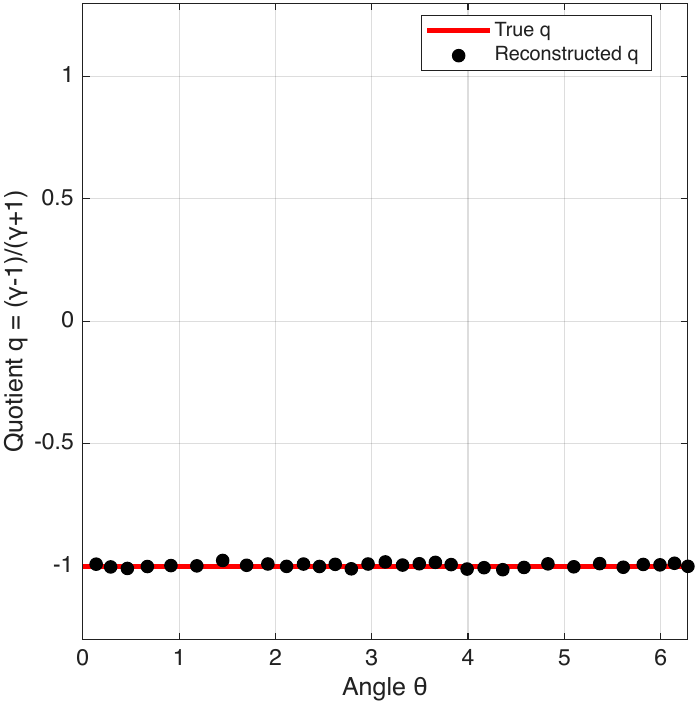}
\includegraphics[width=0.22\textwidth]{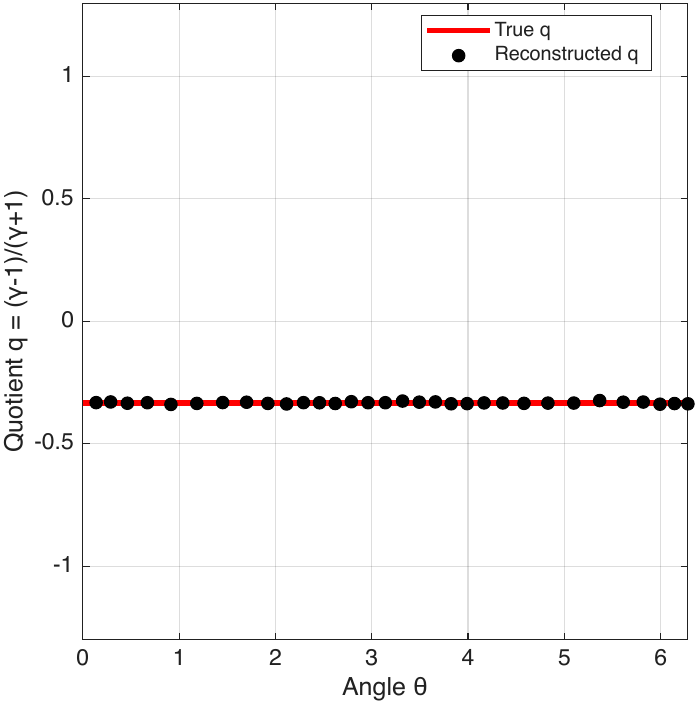}
\includegraphics[width=0.22\textwidth]{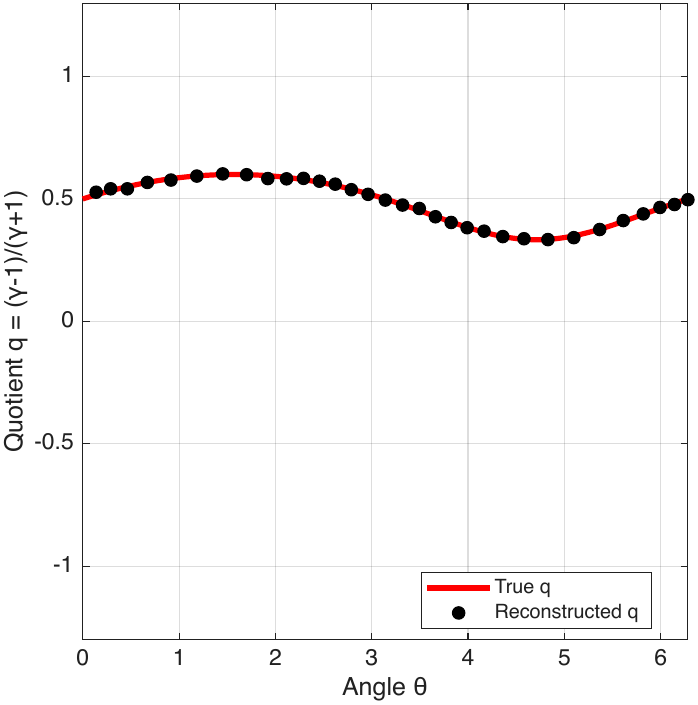}
\includegraphics[width=0.22\textwidth]{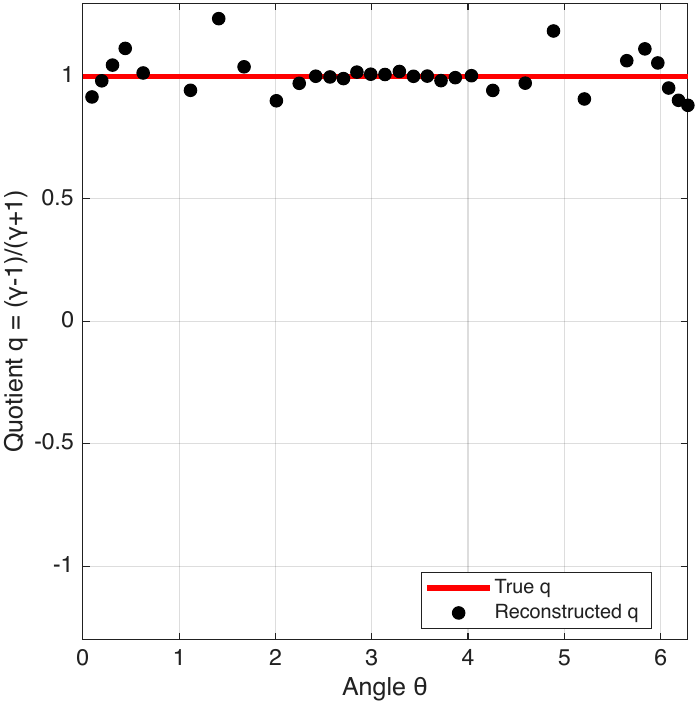}
\includegraphics[width=0.22\textwidth]{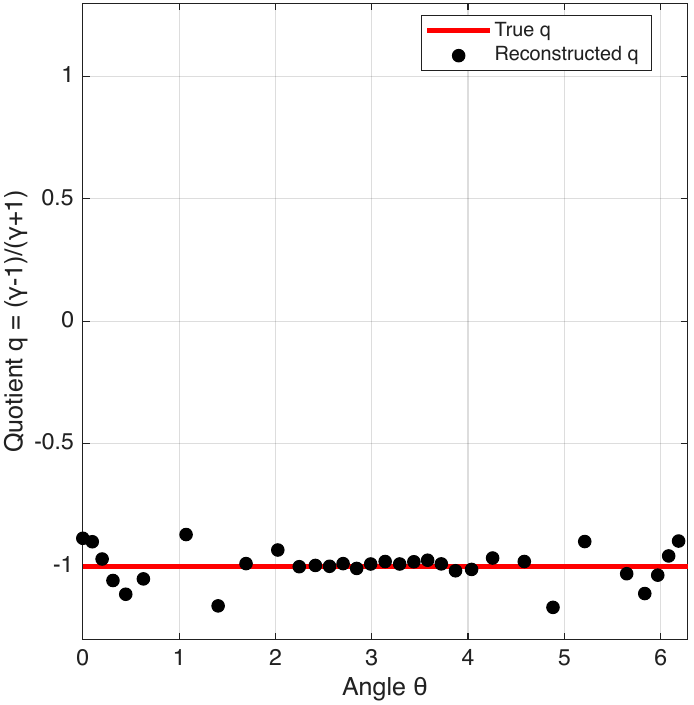}
\includegraphics[width=0.22\textwidth]{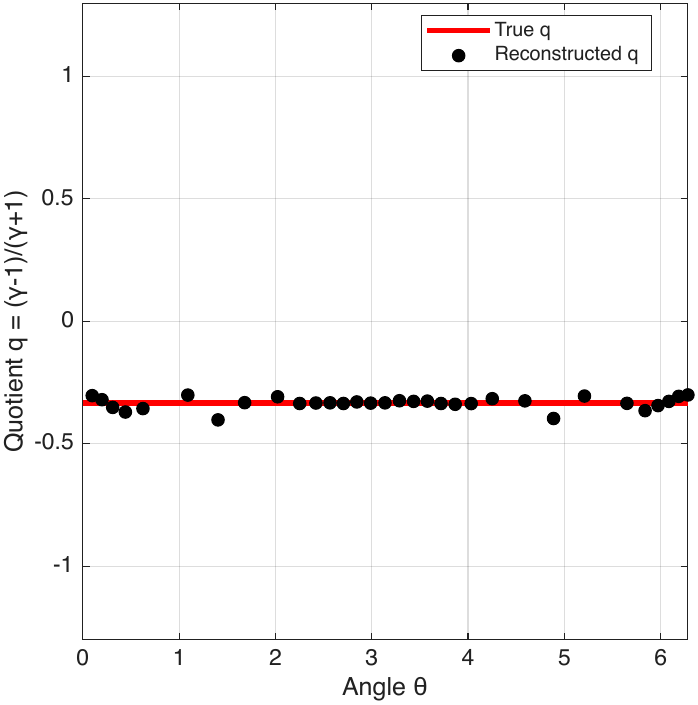}
\includegraphics[width=0.22\textwidth]{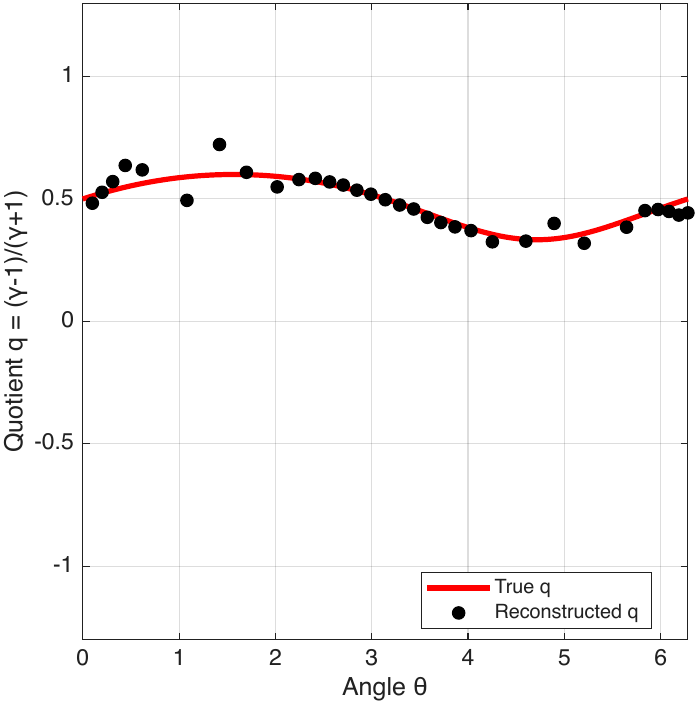}
\caption{Reconstruction of the quotient $q = (\gamma-1)/(\gamma+1)$. Top row: $q$ extracted using the exact ideal boundary. Bottom row: $q$ extracted using the optimized iterative boundary. Columns (left to right): Dirichlet ($q \equiv 1$), Neumann ($q \equiv -1$), constant Robin, and variable Robin. The algorithm reliably identifies the boundary types.}
\label{fig:quotient_reconstruction}
\end{figure}
\subsection{Decoupled boundary condition reconstruction: Step 3}

With the optimized boundary geometry obtained, we proceed to recover the boundary conditions. The reconstruction heavily relies on the boundary shape $\partial D$.
The reconstruction of the quotient $q$ is shown in Figure~\ref{fig:quotient_reconstruction}. The first row presents results using the true boundary, and the second row presents results using the optimized boundary from the previous step. 
When using the true boundary (an idealized scenario), the recovered $q$ values strictly align with the theoretical values: $q \approx 1$ for Dirichlet and $q \approx -1$ for Neumann. Notably, when using the optimized boundary, the recovered $q$ values (black dots) still exhibit excellent agreement with the theoretical ground truth (red lines). The geometric errors introduced in Step 2 only cause negligible fluctuations in $q$, demonstrating that our shape-impedance decoupling approach successfully mitigates the degradation of physical parameters due to small geometric deviations.

For the two Robin cases (where $-1 < q < 1$), the discrete impedance points are explicitly recovered via algebraic inversion.  As depicted in Figure~\ref{fig:impedance_reconstruction}, to suppress point-wise variance, we use the Legendre polynomial approximation (truncated up to degree 5) to fit these recovered points (black dots). The resulting continuous impedance curves (blue dashed lines) accurately capture both the constant magnitude ($\gamma = 0.5$) and the periodic oscillation ($\gamma(t) = 3 + \sin(t)$) of the target functions.

\begin{figure}[htbp]
\centering
\includegraphics[width=0.22\textwidth]{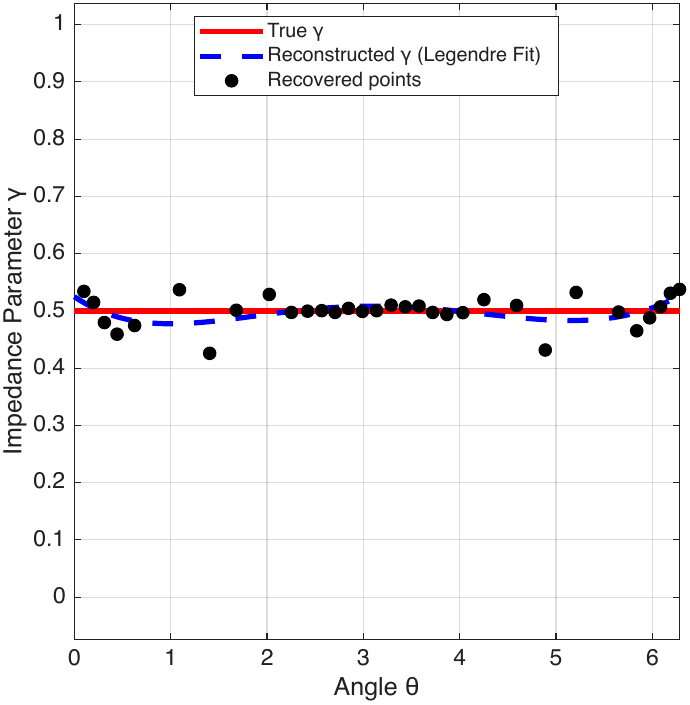}
\includegraphics[width=0.22\textwidth]{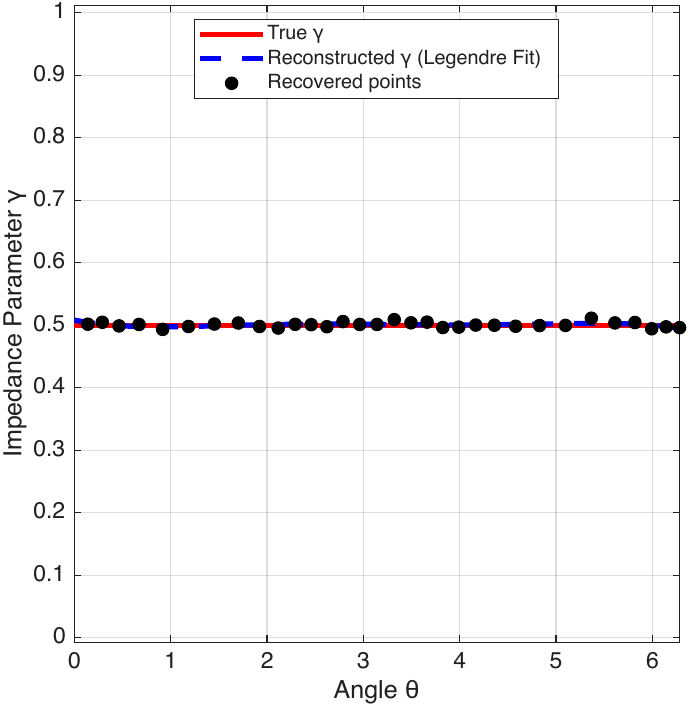}
\includegraphics[width=0.22\textwidth]{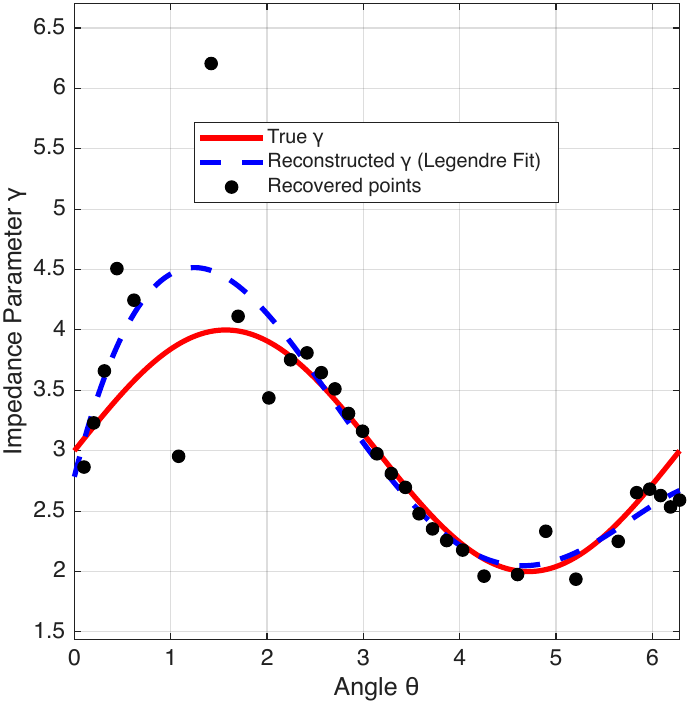}
\includegraphics[width=0.22\textwidth]{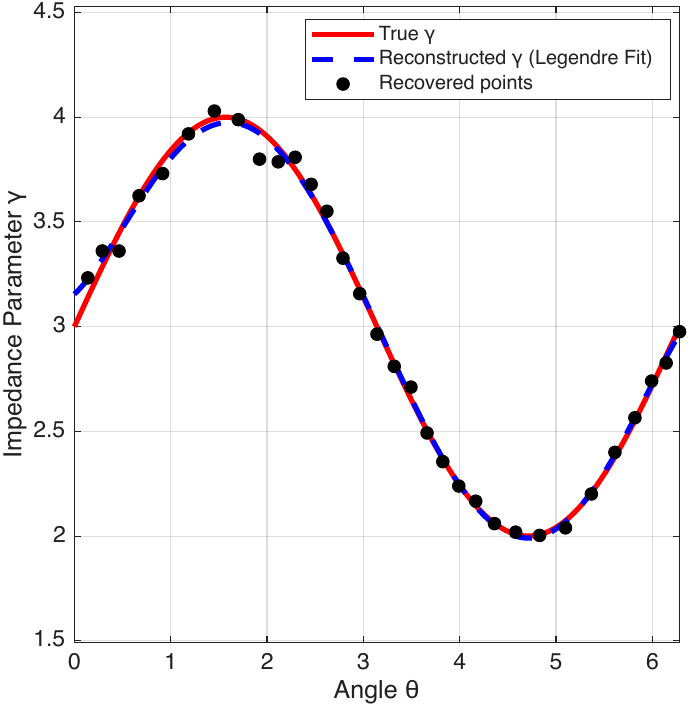}
\caption{Reconstruction of the impedance function $\gamma$. From left to right: Constant impedance $\gamma \equiv 0.5$ with iterated boundary, constant impedance with true boundary, variable impedance $\gamma(t) = 3 + \sin(t)$ with iterated boundary, and variable impedance with true boundary, respectively.} 
\label{fig:impedance_reconstruction}
\end{figure}

Finally, we emphasize that our algorithm entirely avoids solving any forward scattering problems, rendering the proposed method highly efficient and fast to implement. Furthermore, the reconstruction quality can be further improved by employing iterative refinement techniques as described in \cite{KressRundell2018}. We plan to investigate this extension in a future work.

\section*{Acknowledgement}
The research of X. Liu is supported by the National Key R\&D Program of China under grant 2023YFA1009300  and the NNSF of China under grant 12371430.

 
\bibliographystyle{SIAM}

\end{document}